\documentclass[11pt]{article}

\usepackage{IEEEtrantools}
\usepackage{latexsym}
\usepackage{amsfonts}
\usepackage{amssymb}
\usepackage{amsmath,amsfonts,amsthm}
\usepackage{mathrsfs}
\usepackage[all]{xy}
\usepackage{epic}
\usepackage{eepic}
\usepackage{enumerate}
\usepackage{enumitem}
\usepackage{multirow}
\usepackage[breaklinks]{hyperref}

\usepackage{mathrsfs}
\usepackage[all]{xy}
\usepackage{tikz-cd}
\usepackage{epic}
\usepackage{eepic}
\usepackage{multirow}
\usepackage[breaklinks]{hyperref}
\usepackage{tikz}
\usetikzlibrary{matrix,arrows,positioning,calc,decorations.pathmorphing}
\usepackage{color}
\usepackage{bbm}

\newcommand{\bbbt}{\mathbb{T}}
\newcommand{\scrt}{\mathscr{T}}

\newcommand{\be}{\begin{equation}}
	\newcommand{\ee}{\end{equation}}
\newcommand{\bea}{\begin{eqnarray}}
	\newcommand{\eea}{\end{eqnarray}}
\newcommand{\bean}{\begin{eqnarray*}}
	\newcommand{\eean}{\end{eqnarray*}}
\newcommand{\brray}{\begin{array}}
	\newcommand{\erray}{\end{array}}
\newcommand{\biearray}{\begin{IEEEarray}{rCl}}
	\newcommand{\eiearray}{\end{IEEEarray}}

\newtheorem{dfn}{Definition}[section]
\newtheorem{thm}[dfn]{Theorem}
\newtheorem{lmma}[dfn]{Lemma}
\newtheorem{ppsn}[dfn]{Proposition}
\newtheorem{crlre}[dfn]{Corollary}
\newtheorem{xmpl}[dfn]{Example}
\newtheorem{rmrk}[dfn]{Remark}
\newtheoremstyle{case}{}{}{}{}{}{:}{ }{}
\theoremstyle{case}

\AtBeginEnvironment{proof}{\setcounter{case}{0}}

\newcommand{\bdfn}{\begin{dfn}\rm}
	\newcommand{\bthm}{\begin{thm}}
		\newcommand{\blmma}{\begin{lmma}}
			\newcommand{\bppsn}{\begin{ppsn}}
				\newcommand{\bcrlre}{\begin{crlre}}
					\newcommand{\bxmpl}{\begin{xmpl}}
						\newcommand{\brmrk}{\begin{rmrk}\rm}
							
							\newcommand{\edfn}{\end{dfn}}
						\newcommand{\ethm}{\end{thm}}
					\newcommand{\elmma}{\end{lmma}}
				\newcommand{\eppsn}{\end{ppsn}}
			\newcommand{\ecrlre}{\end{crlre}}
		\newcommand{\exmpl}{\end{xmpl}}
	\newcommand{\ermrk}{\end{rmrk}}

\newcommand{\bbz}{\mathbb{Z}}
\newcommand{\bbn}{\mathbb{N}}

\newcommand{\clh}{\mathcal{H}}

\newcommand{\clk}{\mathcal{K}}
\newcommand{\cll}{\mathcal{L}}

\newcommand{\prf}{\noindent{\it Proof\/}: }

\def \qed { \mbox{}\hfill
	$\Box$\vspace{1ex}}

\newcommand\restr[2]{{
		\left.\kern-\nulldelimiterspace 
		#1 
		\littletaller 
		\right|_{#2} 
}}

\newcommand{\littletaller}{\mathchoice{\vphantom{\big|}}{}{}{}}

\numberwithin{equation}{section}
\begin{document}
	
	\author{{\sc Surajit Biswas, Vinay Deshpande, Bipul Saurabh}}
	\title{Crystallization of the quantum Stiefel manifolds $SO_q(2n+1)/SO_q(2n-1)$}

\date{}
\maketitle

\begin{abstract}  
In this paper, we introduce the crystallization $C(SO_0(2n+1)/SO_0(2n-1))$ of the $C^*$-algebra $C(SO_q(2n+1)/SO_q(2n-1))$. We show that every irreducible representation of $C(SO_0(2n+1)/SO_0(2n-1))$ arises as the limit, as $q \to 0^+$, of a family of irreducible representations of $C(SO_q(2n+1)/SO_q(2n-1))$. As a consequence, we prove that $C(SO_0(2n+1)/SO_0(2n-1))$ is the universal $C^*$-algebra generated by a finite set of generators and relations. We then construct a chain of unital essential short exact sequences of $C^*$-algebras, enabling us to compute the $K$-groups of $C(SO_0(2n+1)/SO_0(2n-1))$ and determine explicit generators for them. Using the theory of homogeneous $C^*$-extensions, we prove that $C(SO_0(2n+1)/SO_0(2n-1))$ is isomorphic to $C(SO_q(2n+1)/SO_q(2n-1))$ for every $q \in (0,1)$. 
	\end{abstract}
	{\bf AMS Subject Classification No.:} {\large 58}B{\large 34}, {\large
	46}L{\large 80}, {\large 
	19}K{\large 33}

\section{Introduction}

The theory of crystal bases, introduced independently by Kashiwara (\cite{Kas-1991aa}) and Lusztig (\cite{Lus-1990aa}), describes the limiting behavior of quantum groups as the deformation parameter $q$ approaches $0$. While this theory was originally developed within an algebraic framework, recent years have witnessed significant progress in understanding analogous limiting phenomena at the level of $C^*$-algebras. In this direction, Giri and Pal (\cite{GirPal-2026aa}) introduced the notion of the \emph{crystallization} of quantized function algebras for compact quantum groups of type $A_n$. Subsequently, Giri and Pal (\cite{GirPal-2024aa}) completely classified the irreducible representations of the crystallized algebra $C(SU_0(n+1))$ within their framework, proving that every irreducible representation arises as the limit, as $q \to 0^+$, of a family of irreducible representations of $C(SU_q(n+1))$. This limiting result was subsequently extended by Das and Dey \cite{DasDey-2026aa} to all nondegenerate representations. Soon after, Giselsson (\cite{Gis-2026aa}) proved that the crystallized algebra $C(SU_0(3))$ is isomorphic to a rank-two graph $C^*$-algebra. 
Later, Matassa and Yuncken (\cite{MatYun-2023aa}) developed a general theory of crystallization for compact quantum groups associated with all complex simple Lie algebras, demonstrating that the resulting crystallized algebras are higher-rank graph $C^*$-algebras, with the rank equal to the rank of the underlying Lie algebra. The main difference between these frameworks is that Matassa and Yuncken define an algebra consisting of the matrix coefficients of all irreducible co-representations and then take the limit as $q \to 0^+$, while Giri and Pal take the fundamental representation, take the matrix coefficients of that single representation, take the limit after normalizing appropriately, and then look at the $C^*$-algebra generated by them. Finally, Pal et al. (\cite{DasDeyPal-2026aa}) showed that both the Giri--Pal and Matassa--Yuncken approaches yield the same crystallization for the $A_n$ case.

This  naturally raise the question of whether similar crystallization phenomena occur for quantum homogeneous spaces. In the same paper (\cite{MatYun-2023aa}), 
Matassa and Yuncken  developed a framework for the crystallization of a large family of quantum homogeneous spaces. Let $K$ be a compact connected semisimple Lie group (not necessarily simply connected), let $\Delta$ be a set of simple roots, and let $Y_S$ be the corresponding principal torus bundle over the flag variety $X_S$ associated to a subset $S \subseteq \Delta$. Let $\mathcal{O}[K_0]$ be the quantized function algebra associated with the crystallization $C(K_0)$ of $C(K_q)$. Matassa and Yuncken defined $\mathcal{O}[Y_{S,0}]$ to be the subalgebra of the algebra $\mathcal{O}[K_0]$ generated by certain matrix coefficients. Specifically, if one takes the Lie algebra to be $\mathfrak{so}(2n+1)$, with $\Delta=\{\alpha_1, \dots, \alpha_n\}$ as the set of simple roots and $S=\{\alpha_1, \dots, \alpha_{n-1}\}$, then Definition~$4.9$ of \cite{MatYun-2023aa} yields a crystallization of the quantum Stiefel manifold $SO_q(2n+1)/SO_q(2n-1)$.

In this paper, we focus on the crystallization of the quotient spaces $C(SO_q(2n+1)/SO_q(2n-1))$, 
denoted by $C(SO_0(2n+1)/SO_0(2n-1))$, using an alternative approach. Our treatment differs 
from that of Matassa and Yuncken, aligning instead with the framework established by Giri and Pal. 
We consider the standard generators $\{x_{i,q}\}_{i=1}^{2n+1}$, which comprise the elements of the 
last row of the fundamental (vector) matrix representation of $SO_q(2n+1)$. Specifically, we study 
the Soibelman representation $\psi_q$ of the quantized coordinate ring $C(SO_q(2n+1)/SO_q(2n-1))$ 
acting on a fixed, $q$-independent Hilbert space.
A result due to Soibelman ensures that the operator families $\psi_q(x_{i,q})$ converge in the norm 
topology as $q \to 0^+$ to limiting operators, which we denote by $z_i$. We define the crystallization 
of $C(SO_q(2n+1)/SO_q(2n-1))$ to be the $C^*$-algebra generated by these limiting operators 
$\{z_i\}_{i=1}^{2n+1}$. By employing the Faddeev--Reshetikhin--Takhtajan (FRT) approach, one explicitly 
obtains the commutation relations satisfied by the generators $\{x_{i,q}\}_{i=1}^{2n+1}$. By taking the 
limit as $q \to 0^+$ appropriately, we obtain a new set of relations obeyed by the crystallized 
generators $\{z_i\}_{i=1}^{2n+1}$.
We then define the universal $C^*$-algebra $\mathcal{B}_{2n+1}$ with $2n+1$ generators satisfying these 
limiting relations. By analyzing these relations, we determine all of its irreducible representations. 
Using these representations, we show that $\mathcal{B}_{2n+1}$ is isomorphic to $C(SO_0(2n+1)/SO_0(2n-1))$, 
establishing the latter as a real noncommutative variety. As an important consequence, we prove that all 
irreducible representations of $C(SO_0(2n+1)/SO_0(2n-1))$ arise exactly as $q \to 0^+$ limits of irreducible 
representations of $C(SO_q(2n+1)/SO_q(2n-1))$. 
Using this setup, we construct a faithful representation of $C(SO_0(2n+1)/SO_0(2n-1))$ other 
than the one obtained from its definition. This enables us to obtain a chain of short exact sequences of 
$C^*$-algebras having $C(SO_0(2n+1)/SO_0(2n-1))$ as the middle term. Using this framework, we compute the 
$K$-groups with explicit generators. Using the theory of homogeneous $C^*$-extensions, we prove that 
$C(SO_0(2n+1)/SO_0(2n-1))$ is isomorphic to the $C^*$-algebra $\Sigma^{2(n-1)} \Sigma_2^2 \Sigma^{2(n-1)} C(\mathbb{T})$, 
and hence to $C(SO_q(2n+1)/SO_q(2n-1))$ for every $q \in (0,1)$ (see \cite{Sau-2026aa}).

Several natural questions arise from these results. For instance, can these notions be generalized to the homogeneous spaces $C(SO_q(2n)/SO_q(2n-2))$, and do the approaches of Matassa and Yuncken yield isomorphic crystallizations? Furthermore, since the defining relations of the crystal limit $C(SO_0(2n+1)/SO_0(2n-1))$ yield a canonical $\mathbb{T}^{2n+1}$-action, it is natural to ask whether one can construct a finitely summable, $K$-homologically nontrivial spectral triple that is equivariant under this action. Constructing such spectral triples should be significantly more tractable than in the generic $q$-deformed setting, as the corresponding generators reduce to partial isometries.  We shall address these investigations in subsequent articles.

In the next section, we begin with an overview of the quantum group $SO_q\left(\mathcal{N}\right)$ using the FRT approach (\cite{KliSch-1997aa}), and then recall the quotient space $SO_q(2n+1)/SO_q(2n-1)$. We then define its crystallization, $SO_0(2n+1)/SO_0(2n-1)$. In Section $3$, we find all irreducible representations of $C(SO_0(2n+1)/SO_0(2n-1))$ and establish that it is given by a finite set of generators and relations. We further discuss its primitive ideal space. In the final section, we prove that the $C^*$-algebras of $SO_q(2n+1)/SO_q(2n-1)$ and its crystallization are isomorphic, extending the main result of \cite{Sau-2026aa} to the interval $[0,1)$ by including $q=0$ case.

\textbf{Notations:}
Let $\mathbb{T}$ denote the unit circle in the complex plane, and let $q$ be a real number lying in the interval $(0,1)$. Let  $\mathbb{N}_{0}=\mathbb{N}\cup \{0\}$. The standard bases of the Hilbert spaces $\ell^2(\mathbb{N}_0)$ and
$\ell^2(\mathbb{Z})$ will be denoted by $\{e_m : m \in \mathbb{N}_0\}$ and $\{e_m : m \in \mathbb{Z}\}$ respectively. The length of a Weyl word $w$ is denoted by $\ell(w)$. The number operator $e_m\mapsto me_m$ is denoted by $N$. The operator $S$ denotes the left shift  $e_m\mapsto e_{m-1}$. The $C^{\ast}$-subalgebra of $\cll(\ell^2(\mathbb{N}_0))$ generated by $S$ is denoted by $\scrt$. 
Let $\sigma: \scrt\longrightarrow\mathbb{C} $
be the homomorphism for which $\sigma(S)=1$. For a Hilbert space $H$, the algebra of compact operators on $H$ will be denoted by $\mathcal{K}(H)$, or simply $\mathcal{K}$ when the Hilbert space is clear from the context. For $i,j\in \mathbb{N}_0$, the operator $p_{ij}$ denotes the operator on $\ell^2(\mathbb{N}_0)$ which sends the basis element $e_i$ to $e_j$. 

\section{Preliminaries}
We recall the necessary background and define the crystallization of $SO_q(2n+1)/SO_q(2n-1)$. We then construct a family of irreducible representations of the crystallized $C^*$-algebra.
	
\subsection{The $C^*$-algebra $C(SO_q(\mathcal{N}))$}
Let $\mathcal{N}\geq 3$. We briefly recall the notation and basic relations for the standard Hopf $*$-algebra $\mathcal{O}\left(SO_q\left(\mathcal{N}\right)\right)$, with fundamental matrix $V=(v_j^i)_{i,j=1}^{\mathcal{N}}$, and its universal $C^*$-completion $C\left(SO_q\left(\mathcal{N}\right)\right)$. We follow the conventions of \cite{BhuBisSau-2026aa} and \cite[Section 9.3.3]{KliSch-1997aa}.

Put $i'=\mathcal{N}+1-i$, and define $\rho_i=\mathcal{N}/2-i$ for $i<i'$ and $\rho_{i'}=-\rho_i$ for $i\leq i'$. With $\delta_{ij}$ denoting the Kronecker delta, set $C_j^i=\delta_{ij'}q^{-\rho_i}$ and
$$
R_{mn}^{ij}=
\begin{cases}
(q-q^{-1})
\bigl(\delta_{jm}\delta_{in}-C_i^j C_n^m\bigr),
& i>m,\\[2pt]
q^{\delta_{ij}-\delta_{ij'}}
\delta_{im}\delta_{jn},
& i\leq m.
\end{cases}
$$
The generators $\left\{v_j^i\right\}_{i,j=1}^{\mathcal{N}}$ satisfy
\begin{equation}\label{Relation 2.1}
\sum_{k,l=1}^{\mathcal{N}}
R_{kl}^{ji}v_s^k v_t^l
=
\sum_{k,l=1}^{\mathcal{N}}
R_{st}^{lk}v_k^i v_l^j,
\qquad 1\leq i,j,s,t\leq\mathcal{N},
\end{equation}
and
\begin{equation}\label{Relation 2.2}
VV^*=V^*V=I.
\end{equation}
The involution and comultiplication are determined by
$$
\left(v_j^i\right)^*
=
q^{\rho_i-\rho_j}v_{j'}^{i'},
\qquad
\Delta\left(v_j^i\right)
=
\sum_{k=1}^{\mathcal{N}}v_k^i\otimes v_j^k.
$$
The pair $\left(C\left(SO_q\left(\mathcal{N}\right)\right),\Delta\right)$ is the standard $q$-deformation of $SO\left(\mathcal{N}\right)$. We use the conventions
$C(SO_q(1))=\mathbb{C}$ and $C(SO_q(2))=C(\mathbb{T})$.

\subsection{Crystallization of the Quantum Stiefel Manifold $C\left(SO_q\left(2n+1\right)/SO_q\left(2n-1\right)\right)$}

Let $\left(u_j^i\right)_{1 \leq i,j \leq \mathcal{N}-2}$ denote the generating matrix of
$C\left(SO_q(\mathcal{N}-2)\right)$. Define the $^*$-epimorphism 
$$\eta_{\mathcal{N}} : C\left(SO_q\left(\mathcal{N}\right)\right)
\longrightarrow C\left(SO_q\left(\mathcal{N}-2\right)\right)$$ 
 by
$$
\eta_{\mathcal{N}}\left(v_j^i\right)=
\begin{cases}
u_{j-1}^{i-1}, & 2 \leq i,j \leq \mathcal{N}-1, \\[4pt]
\delta_{ij}, & \text{otherwise}.
\end{cases}
$$

\begin{dfn}\rm(\cite[Section~3]{BhuBisSau-2026aa})
Let $\mathcal{N} \geq 3$. The quantum Stiefel manifold
$SO_q\left(\mathcal{N}\right)/SO_q\left(\mathcal{N}-2\right)$ is defined by
$$
C\left(SO_q\left(\mathcal{N}\right)/SO_q\left(\mathcal{N}-2\right)\right)
=
\left\{
a \in C\left(SO_q\left(\mathcal{N}\right)\right) :
\left(\eta_{\mathcal{N}} \otimes \mathrm{id}\right)\Delta(a)
= I \otimes a
\right\}.
$$
\end{dfn}
The following theorem gives a concrete description of
$C\left(SO_q\left(\mathcal{N}\right)/SO_q\left(\mathcal{N}-2\right)\right)$ in terms of the generators of $C\left(SO_q\left(\mathcal{N}\right)\right)$.

\begin{thm} \rm(\cite[Theorem 3.5]{BhuBisSau-2026aa}) Let $\mathcal{N} \geq 3$. 
Then the  $C^*$-algebra underlying the quotient space
$SO_q\left(\mathcal{N}\right)/SO_q\left(\mathcal{N}-2\right)$
is the $C^*$-subalgebra of  $C\left(SO_q(\mathcal{N})\right)$ generated by
$\left\{v_i^{\mathcal{N}} : 1 \leq i \leq \mathcal{N} \right\}$.
\end{thm}

From now on, we restrict ourselves to the case where $\mathcal{N}$ is odd, and write $\mathcal{N}=2n+1$. 
From  \cite[Corollaries~3.8]{BhuBisSau-2026aa}, we get  a faithful representation of
$C\left(SO_q\left(2n+1\right)/SO_q\left(2n-1\right)\right)$ on the Hilbert space $\ell^2\left(\mathbb{Z}\right) \otimes
\ell^2\left(\mathbb{N}_0\right)^{\otimes (2n-1)}$. We denote this representation by $\psi_{q,n}$, or simply by $\psi_q$ when $n$ is clear from the context.

For $i \in \left\{1,2,\ldots,2n+1\right\}$, define $x_i = v_{i'}^{2n+1}$. When it is necessary to emphasize the dependence on $q$, we write $x_{i,q}$ for $x_i$. From the diagram representation of the faithful representation of $C\left(SO_q\left(2n+1\right)/SO_q\left(2n-1\right)\right)$ given in \cite{BhuSau-2023aa}, one can get the explicit form of each $\psi_q\left(x_{i,q}\right)$, and using that, it is easy to check that  
\begin{enumerate}
	\item $\psi_{q,n}(C\left(SO_q\left(2n+1\right)/SO_q\left(2n-1\right)\right))$ is a $C^*$-subalgebra of $C\left(\mathbb{T}\right)\otimes \mathscr{T}^{\otimes (2n-1)}$.
	\item Moreover,  $\lim_{q \to 0^+} \psi_q\left(x_{i,q}\right)$ exists in $C\left(\mathbb{T}\right)\otimes \mathscr{T}^{\otimes (2n-1)}$.
\end{enumerate}
For $1\leq i \leq 2n+1$, we define 
$$z_i := \lim_{q \to 0^+} \psi_q\left(x_{i,q}\right).$$ 
 Motivated by \cite{GirPal-2024aa}, we then introduce the following notion.

\begin{dfn}
The $C^*$-algebra underlying the crystallization of the quantum Stiefel manifold
$SO_q\left(2n+1\right)/SO_q\left(2n-1\right)$
is defined as the $C^*$-subalgebra of $C\left(\mathbb{T}\right)\otimes \mathscr{T}^{\otimes (2n-1)}$ generated by the elements $z_1, z_2, \ldots, z_{2n+1}$.  It is denoted by $C\left(SO_0\left(2n+1\right)/SO_0\left(2n-1\right)\right)$.
\end{dfn}

In what follows, we describe a family of irreducible representations of the $C^*$-algebra 
$C\big(SO_0(2n+1)/SO_0(2n-1)\big)$. 
These are obtained from irreducible representations of its $q$-analogs 
by letting $q$ approach $0$.
For that, let $\rho : \mathscr{T} \to C\left(\mathbb{T}\right)$ be the surjective $^*$-homomorphism
sending the generator $S$ to the function $\mathbf{t}$. For each $t\in\mathbb{T}$, let $\mathrm{ev}_t : C\left(\mathbb{T}\right) \to \mathbb{C}$ denote the evaluation homomorphism defined by $\mathrm{ev}_t(f)=f(t)$ for
$f\in C\left(\mathbb{T}\right)$, and define the map
$\sigma_t : \mathscr{T} \to \mathbb{C}$ by $\sigma_t := \mathrm{ev}_t \circ \rho$. Observe that $\sigma_t(S)=t$. We write $\sigma$ for the map $\sigma_1$. 

For $k \in \mathbb{N}$, let $\sigma_{k-1}^{k-1}$ denote the identity homomorphism on
$C(\mathbb{T}) \otimes \mathscr{T}^{\otimes (k-1)}$. Define the $^*$-homomorphism
$$\sigma_{k-1}^k : C\left(\mathbb{T}\right)\otimes \mathscr{T}^{\otimes k}
\to C\left(\mathbb{T}\right)\otimes \mathscr{T}^{\otimes (k-1)};\,\,  \quad \sigma_{k-1}^k
:= 1 \otimes 1^{\otimes (k-1)} \otimes \sigma.$$ 
For $m,k\in\mathbb{N}_0$ with $m < k$, we define 
$$\sigma_m^k
= \sigma_m^{m+1} \circ \sigma_{m+1}^{m+2} \circ \cdots \circ \sigma_{k-1}^k.$$ 
Moreover, for $t\in\mathbb{T}$ and for $m,k\in\mathbb{N}_0$ with $m\leq k$, we define 
$$\sigma_{(t,m)}^k
= \left(\mathrm{ev}_t \otimes 1^{\otimes m}\right) \circ \sigma_m^k.$$ 
Using the natural embedding
$
\mathscr{T}^{\otimes m}
\subseteq
\mathcal{L}\left(\ell^2(\mathbb{N}_0)^{\otimes m}\right)
$
and  \cite[Theorem 3.6]{BhuBisSau-2026aa} and \cite[Theorem 2.2]{NesTus-2012aa}, it follows that, for each $t\in\mathbb{T}$ and $0\leq m\leq 2n-1$, the representation
$$
\sigma_{(t,m)}^{2n-1}\circ\psi_q
\colon C\left(SO_q(2n+1)/SO_q(2n-1)\right)
\longrightarrow
\mathcal{L}\left(\ell^2(\mathbb{N}_0)^{\otimes m}\right)
$$
is irreducible. Moreover, every irreducible representation of this $C^*$-algebra is unitarily equivalent to one of these representations.
When $n$ is fixed and $0\leq m\leq 2n-1$, we simplify the notation by writing
$$
\sigma_{t,m}
:=
\sigma_{(t,m)}^{2n-1},
\quad \text{and} \quad
\sigma_{t,m}^{(q)}
:=
\sigma_{(t,m)}^{2n-1}\circ\psi_q.
$$

\brmrk \label{Indexing of the irreducible representations}
For each $t\in\mathbb{T}$ and $0\leq m\leq 2n-1$, and via the natural embedding
$\mathscr{T}^{\otimes m} \subseteq \mathcal{L}\left(\ell^2\left(\mathbb{N}_0\right)^{\otimes m}\right)$, the irreducible representation
$\sigma_{t,m}^{(q)}$ coincides with the representation appearing in the classification of irreducible representations of $C(SO_q(2n+1)/SO_q(2n-1))$ in \cite{BhuBisSau-2026aa}, namely, $\sigma_{t,m}^{(q)}
= \phi_{t,\omega_{m+1}}$.
Using the diagram representations for $\phi_{t,\omega_{m+1}}$ (see \cite{BhuSau-2023aa}), we obtain:
\begin{itemize}
\item If $0\leq m\leq n-1$, then
$$\sigma_{t,m}^{(q)}(x_r)=0
\quad \text{for } r>m+1,\quad \text{and} \quad \sigma_{t,m}^{(q)}\left(x_{m+1}\right)
= (-1)^m\, t\, (q^{2N})^{\otimes m}.$$

\item If $n-1 < m \leq 2n-1$, then
$$\sigma_{t,m}^{(q)}\left(x_r\right)=0
\quad \text{for } r>m+2, \quad \text{and} \quad \sigma_{t,m}^{(q)}\left(x_{m+2}\right)
= (-1)^{n-1}\, t\, (q^{2N})^{\otimes m}.$$
\end{itemize}
Consequently, letting $q\to 0^+$ in the above expressions, we obtain the following.

\begin{itemize}
\item If $0\leq m\leq n-1$, then 
$$\sigma_{t,m}\left(z_r\right)=0 \quad \text{for } r>m+1, \quad \text{and} \quad \sigma_{t,m}\left(z_{m+1}\right) = (-1)^m tp^{\otimes m}.$$

\item If $n-1 < m \leq 2n-1$, then 
$$\sigma_{t,m}\left(z_r\right)=0 \quad \text{for } r>m+2, \quad \text{and} \quad \sigma_{t,m}\left(z_{m+2}\right) = (-1)^{n-1}\, tp^{\otimes m}.$$
\end{itemize}
\ermrk 

\begin{lmma}\label{Inclusion of compacts}
Let $t\in\mathbb{T}$, and let $0\leq m\leq 2n-1$. Then
$$\mathcal{K}\left(\ell^2\left(\mathbb{N}_0\right)^{\otimes m}\right) \subset \sigma_{t,m}\left(C\left(SO_0\left(2n+1\right)/SO_0\left(2n-1\right)\right)\right).$$
\end{lmma}

\begin{proof}
We choose $1\leq j\leq 2n+1$ such that
$$
\sigma_{t,m}^{(q)}\left(x_{j,q}\right)\neq 0
\quad\text{and}\quad
\sigma_{t,m}^{(q)}\left(x_{r,q}\right)=0
\ \text{for all}\ j<r\leq 2n+1.
$$
By Remark \ref{Indexing of the irreducible representations}, $\sigma_{t,m}^{(q)}\left(x_{j,q}^*x_{j,q}\right)
= \left(q^{4N}\right)^{\otimes m}$. Taking the limit as $q\to 0^+$, we obtain $\sigma_{t,m}\left(z_j^*z_j\right)
=
p^{\otimes m}$. Thus
$$
p^{\otimes m}
\in
B_{t,m}
:=
\sigma_{t,m}\left(C\left(SO_0\left(2n+1\right)/SO_0\left(2n-1\right)\right)\right).
$$

By \cite[Lemma 4.3]{BhuSau-2023aa}, for each
$j_1,\ldots,j_m\in\mathbb{N}_0$, there exists a polynomial
$P_{j_1,\ldots,j_m}^{(q)}$ in the non-commuting variables
$\sigma_{t,m}^{(q)}\left(x_{1,q}\right),\ldots,
\sigma_{t,m}^{(q)}\left(x_{2n+1,q}\right)$ and their adjoints such that
$$
P_{j_1,\ldots,j_m}^{(q)}\left(e_0^{\otimes m}\right)
=
e_{j_1}\otimes\cdots\otimes e_{j_m},
$$
that is,
$$
P_{j_1,\ldots,j_m}^{(q)}\circ p^{\otimes m}
=
p_{j_1 0}\otimes\cdots\otimes p_{j_m 0}.
$$

Taking the limit as $q\to 0^+$, we obtain a polynomial
$P_{j_1,\ldots,j_m}$ in the variables
$\sigma_{t,m}\left(z_1\right),\ldots$,
$\sigma_{t,m}\left(z_{2n+1}\right)$
and their adjoints such that
$$
P_{j_1,\ldots,j_m}\circ p^{\otimes m}
=
p_{j_1 0}\otimes\cdots\otimes p_{j_m 0}.
$$

Since $P_{j_1,\ldots,j_m}$ and $p^{\otimes m}$ belong to $B_{t,m}$, it follows that $p_{j_1 0}\otimes\cdots\otimes p_{j_m 0}
\in
B_{t,m}$. Hence, for each
$i_1,\ldots,i_m,j_1,\ldots,j_m\in\mathbb{N}_0$, we have $p_{j_1 i_1}\otimes\cdots\otimes p_{j_m i_m}
\in
B_{t,m}$. Therefore, $\mathcal{K}\left(\ell^2\left(\mathbb{N}_0\right)^{\otimes m}\right)
\subset
B_{t,m}$.
\end{proof}

From now on, we use the same notation $\sigma_{t,m}$ for its restriction to $C\left(SO_0(2n+1)/SO_0(2n-1)\right)$ whenever the domain $C\left(SO_0(2n+1)/SO_0(2n-1)\right)$ is clear from the context.

\begin{ppsn}\label{Natural irr. repn}
Let $t\in\mathbb{T}$ and let $0\leq m \leq 2n-1$. Then the restriction
$$
\sigma_{t,m} : C\left(SO_0\left(2n+1\right)/SO_0\left(2n-1\right)\right) \longrightarrow \mathcal{L}\left(\ell^2\left(\mathbb{N}_0\right)^{\otimes m}\right)
$$
is an irreducible representation.
\end{ppsn}

\begin{proof}
Suppose, to the contrary, that $\sigma_{t,m}$ is reducible. Then there exists a nontrivial projection $E\in\mathcal{L}\left(\ell^2\left(\mathbb{N}_0\right)^{\otimes m}\right)$ such that
$E\sigma_{t,m}(z) = \sigma_{t,m}(z)E$ for each $z\in C\left(SO_0\left(2n+1\right)/SO_0\left(2n-1\right)\right)$. By Lemma \ref{Inclusion of compacts}, it then follows that $EK = KE$ for all $K\in \mathcal{K}\left(\ell^2\left(\mathbb{N}_0\right)^{\otimes m}\right)$, which implies that $E = 0$ or $E = 1$, contradicting the nontriviality of $E$. Therefore, our assumption was false, and $\sigma_{t,m}$ is irreducible.
\end{proof}

\brmrk
Observe that for each $t\in\mathbb{T}$,  $0\leq m \leq 2n-1$, and $1\leq i\leq 2n+1$, we have
\[
\sigma_{t,m} \left(z_i\right) = \lim_{q\to 0^+} \sigma_{t,m}^{(q)} \left(x_{i,q}\right).
\]
Therefore, by Proposition \ref{Natural irr. repn}, each irreducible representation
$$\sigma_{t,m}: C\left(SO_0\left(2n+1\right)/SO_0\left(2n-1\right)\right)\rightarrow \mathcal{L}\left(\ell^2\left(\mathbb{N}_0\right)^{\otimes m}\right)$$ arises as the $q \rightarrow 0^+$ limit of the irreducible representations $\sigma_{t,m}^{(q)} : C\left(SO_q\left(2n+1\right)/SO_q\left(2n-1\right)\right)\rightarrow \mathcal{L}\left(\ell^2\left(\mathbb{N}_0\right)^{\otimes m}\right)$.
It is then natural to ask whether the collection
$\sigma_{t,m}$, $t\in\mathbb{T}$, $0\leq m\leq 2n-1$,
exhausts, up to unitary equivalence, all irreducible representations of
$C\left(SO_0\left(2n+1\right)/SO_0\left(2n-1\right)\right)$. This will be answered affirmatively in Corollary \ref{Characterization of all irr. repn odd}.
\ermrk

\section{Irreducible representations of $C\left(SO_0(2n+1)/SO_0(2n-1)\right)$}

In this section, we derive relations for the generators of $C(SO_q(2n+1)/SO_q(2n-1))$ and obtain their limiting form as $q\to0^+$. We then show that $C(SO_0(2n+1)/SO_0(2n-1))$ is the universal $C^*$-algebra determined by these limiting relations. We also describe the primitive ideal spaces of certain quotient algebras $\mathcal{B}_{2n+1}^{k}$ associated with a natural filtration of this algebra by ideals.
We begin with the following relations, derived as in \cite[Proposition~14(ii)]{KliSch-1997aa}.

\begin{ppsn}\label{Odd q relation}
In the $C^*$-algebra
$$
C\left(SO_q(2n+1)/SO_q(2n-1)\right) 
= C^*\left(\left\{x_i : 1\leq i \leq 2n+1\right\}\right),
$$
the generators $\left\{x_i : 1 \leq i \leq 2n+1\right\}$ satisfy the following relations:
\begin{align}
x_i x_j&=q x_j x_i & \text{for } & i > j,  i+j \neq 2n+2 \label{c1}\\
x_i x_{i'}  &= q^2 x_{i'} x_i - (1-q^2)\sum_{k > i} q^{i-k} x_k x_{k'} & \text{for } &  i > n+1 \label{c2}\\
x_i^* x_{i'} &= q^2 x_{i'} x_i^* & \text{for } & i \neq n+1 \label{c3}\\
x_i^* x_j &= q x_j x_i^* & \text{for }  & i+j>2n+2,  i \neq j \label{c4}\\
x_i^* x_j &= q x_j x_i^* - (1-q^2)q^{\rho_j + \rho_i} x_{i'}x_{j'}^* 
& \text{for }  & i+j<2n+2,  i \neq j \label{c5}\\
x_i^* x_i &= q^{\delta_{ii'}} x_i x_i^* + (1-q^2)\sum_{k > i} x_k x_k^* & \text{ for } & i \geq n+1 \label{c6}\\
x_i^* x_i &= x_i x_i^* - (1 -q^2)q^{2\rho_i} x_{i'} x_{i'}^* + (1-q^2)\sum_{k > i} x_k x_k^* & \text{for } & i < n+1 \label{c7}\\
\sum_{i=1}^{2n+1}&x_i x_i^* =1 \label{c8}
\end{align}
\end{ppsn}

\begin{proof}
In relation \eqref{Relation 2.1}, taking $i=j=2n+1$, we obtain
\begin{equation}\label{General product relation 1}
qx_s x_t
=
\sum_{k,l=1}^{2n+1} R_{s't'}^{l'k'} x_k x_l, \qquad 1\leq s,t \leq 2n+1.
\end{equation}

Similarly, rewriting \eqref{Relation 2.1} as 
\begin{align*}
\sum_{k,l \geq 1} R_{kl}^{ji} q^{\rho_k - \rho_s} \left(v_{s'}^{k'}\right)^* v_{t'}^l &= \sum_{k,l \geq 1} R_{st'}^{lk} q^{\rho_j - \rho_l} v_k^i \left(v_{l'}^{j'}\right)^*,
\end{align*}

and choosing $i=2n+1$, $j=1$, we obtain
\begin{equation}\label{General product relation 2}
x_s^* x_t = q^{1+\rho_s} \sum_{k,l =1}^{2n+1} R_{st'}^{lk'} q^{- \rho_l} x_k x_l^*, \qquad 1\leq s,t \leq 2n+1.
\end{equation}

We now derive each relation.

\begin{enumerate}[label=(\arabic*)]
\item Let $s>t$ and $s+t\neq 2n+2$. Then $t\neq s'$, so $C_{t'}^{s'}=0$, and
$$
R_{s't'}^{l'k'}
=
\begin{cases}
\left(q-q^{-1}\right)\delta_{ks}\delta_{lt}, & \text{if } l<s,\\
q^{\delta_{lk}-\delta_{l'k}}\delta_{ls}\delta_{kt}, & \text{if } l\geq s.
\end{cases}
$$

Substituting into \eqref{General product relation 1} gives 
\begin{align*}
q x_s x_t
&=
\sum_{\substack{k,l=1\\ l<s}}^{2n+1}
\left(q-q^{-1}\right)\delta_{ks}\delta_{lt} x_k x_l
+
\sum_{k=1}^{2n+1}
q^{\delta_{sk}-\delta_{s'k}}\delta_{kt} x_k x_s\\
&=
\sum_{k=1}^{2n+1}
\left(q-q^{-1}\right)\delta_{ks} x_k x_t
+
q^{\delta_{st}-\delta_{s't}} x_t x_s\\
&=
\left(q-q^{-1}\right)x_s x_t + x_t x_s.
\end{align*}
Hence $x_s x_t = q x_t x_s$, proving \eqref{c1}.

\item We rewrite equation \ref{General product relation 1} in the form 
$$qx_sx_t = \sum_{k,l=1}^{2n+1} R_{s't'}^{lk'} x_k x_{l'}.$$
Let $s=t'$ and assume $t> n+1$. If $l>t$, then $l' < t' < n+1 < t$, hence $\delta_{l't} =0$. Therefore,
$$R_{tt'}^{lk'} = \begin{cases}
-\left(q - q^{-1}\right)\delta_{kl} q^{-\left(\rho_{k'} + \rho_t\right)}, &\text{if } l>t,\\
q^{\delta_{lk'} - \delta_{lk}} \delta_{lt} \delta_{kt}, &\text{if } l\leq t.
\end{cases}$$
Substituting into the above rewritten form of equation \ref{General product relation 1},
\begin{align*}
q x_{t'} x_t &= -\left(q-q^{-1}\right)\sum_{\substack{k,l=1 \\ l>t}}^{2n+1} \delta_{kl} q^{\rho_{t'} -\rho_{k'}} x_k x_{l'} + \sum_{\substack{k,l =1 \\ l\leq t}}^{2n+1} q^{\delta_{lk'} - \delta_{lk}} \delta_{lt} \delta_{kt} x_k x_{l'}\\
&= -\left(q-q^{-1}\right)\sum_{l>t} q^{\rho_{t'} - \rho_{l'}} x_l x_{l'} + q^{-1} x_t x_{t'}.
\end{align*}
Since
$$
\rho_{t'} - \rho_{l'} = \left(\frac{2n+1}{2} - t'\right) - \left(\frac{2n+1}{2} - l'\right) = t-l \quad \text{for all} \quad l>t,$$
we obtain
$$q x_{t'} x_t = q^{-1}\left(1-q^2\right)\sum_{l>t} q^{t-l} x_l x_{l'} + q^{-1} x_t x_{t'}.$$
Multiplying both sides by $q$, we finally obtain
$$x_t x_{t'} = q^2 x_{t'}x_t - \left(1-q^2\right)\sum_{l>t} q^{t-l} x_l x_{l'},$$
proving \eqref{c2}.

\item Let $t=s'$ with $s\neq n+1$. Since $s\neq s'$, we have $\delta_{ss'} =0$. Therefore,
$$R_{ss}^{lk'} = \begin{cases}
0, &\text{if } l>s,\\
q^{\delta_{lk'} - \delta_{lk}} \delta_{ls} \delta_{k's}, &\text{if } l\leq s.
\end{cases}$$
Substituting into \eqref{General product relation 2},
$$x_s^* x_{s'} = q^{1+\rho_s}\sum_{\substack{k,l=1 \\ l\leq s}}^{2n+1} q^{\delta_{lk'} - \delta_{lk}} \delta_{ls} \delta_{k's} q^{-\rho_l} x_k x_l^* = q^{1+\rho_s}\left(q^{\delta_{ss} - \delta_{ss'} -\rho_s} x_{s'} x_s^*\right) = q^2 x_{s'} x_s^*,$$
proving \eqref{c3}.

\item Assume $s+t> 2n+2$ and $s\neq t$. Then $C_{t'}^s = 0$, and if $l>s$, then $\delta_{lt'} = 0$. Thus
$$R_{st'}^{lk'}= \begin{cases}
0, &\text{if } l>s,\\
q^{\delta_{lk'} - \delta_{lk}} \delta_{ls}\delta_{kt}, &\text{if } l\leq s.
\end{cases}$$
Substituting into \eqref{General product relation 2},
$$x_s^* x_t = q^{1+\rho_s}\left(q^{\delta_{st'} - \delta_{st}} q^{-\rho_s} x_t x_s^*\right) = qx_t x_s^*,$$
proving \eqref{c4}. 

\item Assume $s+t < 2n+2$ and $s\neq t$. Then $C_{t'}^s =0$, and therefore 
$$R_{st'}^{lk'} = \begin{cases}
\left(q- q^{-1}\right) \delta_{k's}\delta_{lt'}, &\text{if } l>s,\\
q^{\delta_{lk'} - \delta_{lk}} \delta_{ls} \delta_{kt}, &\text{if } l\leq s.
\end{cases}$$
Substituting into equation \eqref{General product relation 2} yields
\begin{align*}
x_s^* x_t &= q^{1+\rho_s}\left(\sum_{\substack{k,l =1 \\ l>s}}^{2n+1}\left(q-q^{-1}\right)q^{-\rho_l} \delta_{k's} \delta_{lt'} x_k x_l^* + \sum_{\substack{k,l =1 \\ l\leq s}}^{2n+1} q^{\delta_{lk'} - \delta_{lk}} q^{-\rho_l}\delta_{ls} \delta_{kt} x_k x_l^*\right)\\
&= q^{1+\rho_s}\left(\left(q-q^{-1}\right) q^{-\rho_{t'}} x_{s'} x_{t'}^* + q^{\delta_{st'} - \delta_{st}} q^{-\rho_s} x_t x_s^*\right)\\
&= q x_t x_s^* - \left(1-q^2\right) q^{\rho_t +\rho_s} x_{s'} x_{t'}^*,
\end{align*}
proving \eqref{c5}.

\item Let $t=s\geq n+1$. Then, for each $l>s$, $\delta_{ls'} = 0$. Therefore
$$R_{ss'}^{lk'} = \begin{cases}
-\left(q-q^{-1}\right) \delta_{k'l'} q^{-\left(\rho_{k'} + \rho_s\right)}, &\text{if } l>s,\\
q^{\delta_{lk'} - \delta_{lk}} \delta_{ls} \delta_{ks}, &\text{if } l\leq s.
\end{cases}$$
Substituting into \eqref{General product relation 2} gives
\begin{align*}
x_s^* x_s &= q^{1+\rho_s}\left(-\left(q-q^{-1}\right)\sum_{\substack{k,l=1 \\ l>s}}^{2n+1} \delta_{kl} q^{-\left(\rho_{k'} + \rho_s\right)} q^{-\rho_l} x_k x_l^* + \sum_{\substack{k,l=1 \\ l\leq s}}^{2n+1} q^{\delta_{lk'} - \delta_{lk}} \delta_{ls} \delta_{ks} q^{-\rho_l} x_k x_l^*\right)\\
&= q^{1+\rho_s}\left(-\left(q-q^{-1}\right)\sum_{l>s} q^{-\left(\rho_{l'} + \rho_s\right) -\rho_l} x_l x_l^* + q^{\delta_{ss'} - 1 - \rho_s} x_sx_s^*\right)\\
&= q^{\delta_{ss'}} x_s x_s^* + \left(1-q^2\right) \sum_{l>s} x_l x_l^*,
\end{align*}
proving \eqref{c6}.

\item Let $t=s< n+1$. Then
$$R_{ss'}^{lk'} = \begin{cases}
\left(q-q^{-1}\right)\left(\delta_{k's}\delta_{ls'} - \delta_{kl} q^{-\rho_{k'} - \rho_s}\right), &\text{if } l>s,\\
q^{\delta_{lk'} - \delta_{lk}} \delta_{ls} \delta_{ks}, &\text{if } l\leq s.
\end{cases}$$
Substituting into \eqref{General product relation 2} gives
\begin{align*}
x_s^* x_s &= q^{1+\rho_s}\left(\left(q-q^{-1}\right)\left(\sum_{\substack{k,l=1 \\ l>s}}^{2n+1}\delta_{k's} \delta_{ls'} q^{-\rho_l} x_k x_l^* - \sum_{\substack{k,l=1 \\ l>s}}^{2n+1}\delta_{kl} q^{-\rho_{k'} - \rho_s} q^{-\rho_l} x_k x_l^*\right) + q^{\delta_{ss'} - 1} q^{-\rho_s} x_s x_s^*\right)\\
&= q^{1+\rho_s}\left(\left(q-q^{-1}\right)\left(q^{-\rho_{s'}} x_{s'} x_{s'}^* - \sum_{l>s} q^{-\rho_{l'} - \rho_s - \rho_l} x_l x_l^*\right) + q^{-1-\rho_s} x_s x_s^*\right)\\
&= x_s x_s^* - \left(1-q^2\right) q^{2\rho_s} x_{s'} x_{s'}^* + \left(1-q^2\right)\sum_{l>s} x_l x_l^*,
\end{align*}
proving \eqref{c7}.

\item The $(2n+1, 2n+1)$ entry of $VV^*$ equals $\sum_{i=1}^{2n+1} x_i x_i^*$. Since $VV^* = I$, we obtain $\sum_{i=1}^{2n+1} x_i x_i^* =1$, proving \eqref{c8}. 
\end{enumerate}
\end{proof} 

The elements $z_i=\lim_{q\to 0^+}\psi_q\left(x_{i,q}\right)$ for $i=1,2,\ldots, 2n+1$ satisfy the following relations obtained from \eqref{c1}--\eqref{c8} after taking the limit $q\to 0^+$:
\begin{align}
	z_i z_j &= 0, & & i>j, \label{r1.0}\\
	z_i^* z_j &= 0, & & i\neq j, \label{r2.0}\\
	z_i^* z_i &= z_i z_i^* + \sum_{k>i} z_k z_k^*, & & \forall\, i\neq n+1, \label{r3.0}\\
	z_{n+1}^* z_{n+1} &= \sum_{k> n+1} z_k z_k^*, \label{r4.0}\\
	\sum_{i=1}^{2n+1} z_i z_i^* &= 1. \label{r5.0}
\end{align}
We denote this system of relations by $\mathcal{R}_{2n+1}$. For $1\leq k \leq 2n$, if we additionally impose the constraints $z_{2n+1}=z_{2n}=\dots=z_{k+1}=0$, we denote the resulting set of relations by $\mathcal{R}^k_{2n+1}$. We also set $\mathcal{R}^{2n+1}_{2n+1}=\mathcal{R}_{2n+1}$.

\begin{dfn}
	For $1\leq k \leq 2n+1$, let $\mathcal{B}^k_{2n+1}$ denote the universal $C^*$-algebra generated by the elements $z_1, z_2, \ldots, z_{2n+1}$ subject to the relations $\mathcal{R}^k_{2n+1}$. We write $\mathcal{B}^{2n+1}_{2n+1}$ simply as $\mathcal{B}_{2n+1}$.
\end{dfn}
\brmrk \label{rem1}
By Remark \ref{Indexing of the irreducible representations}, the images of the generators of $C(SO_0(2n+1)/SO_0(2n-1))$ under the representation $\sigma_{t,m}$ satisfy the relations $\mathcal{R}_{2n+1}^{m+1}$ for $0\leq m\leq n-1$, and the relations $\mathcal{R}_{2n+1}^{m+2}$ for $n-1<m\leq 2n-1$. Moreover, relation \eqref{r4.0} gives $\mathcal{R}_{2n+1}^{n+1}=\mathcal{R}_{2n+1}^{n}$. Finally, relation~\eqref{r5.0} implies that every $^*$-representation $\pi$ satisfies $\|\pi\left(z_i\right)\|\leq 1$ for $1\leq i\leq 2n+1$. Hence, the universal $C^*$-algebra $\mathcal{B}_{2n+1}^k$ exists for each $1\leq k\leq 2n+1$. By the universal properties, for each $1\leq k\leq2n+1$
there is a canonical isomorphism
$$
\mathcal{B}_{2n+1}^k
\cong
\mathcal{B}_{2n+1}/
\left\langle z_{k+1},z_{k+2},\ldots,z_{2n+1}\right\rangle.
$$
Here $\langle\cdot\rangle$ denotes the closed two-sided
ideal generated by the indicated elements, and the ideal
is understood to be $\{0\}$ when $k=2n+1$.
\ermrk

\brmrk \label{rem2}
Note that if $1 \leq k \leq n$, then the defining relations $\mathcal{R}^k_{2n+1}$ of $\mathcal{B}^k_{2n+1}$ are exactly the same as those for $C(S_0^{2k-1})$. Hence, mapping the generators $z_i$ to the standard generators of $C(S_0^{2k-1})$ will induce a canonical isomorphism between the two.
\ermrk

We begin by analyzing the $C^*$-algebra $\mathcal{B}_{2n+1}$, and determining all of its irreducible representations. We must point out that many of these results hold analogously for the $C^*$-algebra $\mathcal{B}^k_{2n+1}$.

\begin{ppsn}\label{Extra relation for odd}
In the $C^*$-algebra $\mathcal{B}_{2n+1}$, we have $z_{n+1}^2=0$.
\end{ppsn}

\begin{proof}
We have
\begin{align*}
\left(z_{n+1}^2\right)^* z_{n+1}^2
&= z_{n+1}^*\left(z_{n+1}^* z_{n+1}\right)z_{n+1}\\
&= z_{n+1}^*\left(\sum_{k>n+1} z_k z_k^*\right)z_{n+1}
\tag*{(by relation \eqref{r4.0})}\\
&= \sum_{k>n+1} \left(z_{n+1}^* z_k\right)\left(z_k^* z_{n+1}\right)
= 0
\tag*{(by relation \eqref{r2.0}).}
\end{align*}
Therefore $z_{n+1}^2=0$.
\end{proof}

From relation \eqref{r3.0}, it follows  that $z_{2n+1}$ is normal. For simplicity, we denote $z_{2n+1}^* z_{2n+1}$ as $\omega$.

\begin{lmma}\label{Properties of omega for odd}
In the $C^*$-algebra $\mathcal{B}_{2n+1}$, the following relations hold:
\begin{enumerate}[label=(\arabic*)]
\item $\omega z_i = 0$ for $1 \leq i < 2n+1$,
\item $\sigma(\omega) \subseteq \left\{0,1\right\}$.
\end{enumerate}
\end{lmma}

\begin{proof}
\begin{enumerate}[label=(\arabic*)]
\item For $1 \leq i < 2n+1$, we compute $\omega z_i = z_{2n+1}^*\left(z_{2n+1} z_i\right) = 0$, where the last equality follows from relation \eqref{r1.0}.

\item Since $\omega$ is self-adjoint, it suffices to prove that $\omega$ is a projection, i.e., $\omega^2 = \omega$.  
By relation \eqref{r5.0}, we have
$$
\omega = \left(\sum_{i=1}^{2n+1} z_i z_i^*\right)\omega
       = \sum_{i=1}^{2n} z_i\left(z_i^* z_{2n+1}\right) z_{2n+1}^* + \omega^2.
$$
Using relation \eqref{r2.0}, the summation term vanishes, hence $\omega = \omega^2$. Therefore, $\sigma(\omega) \subseteq \left\{0,1\right\}$.
\end{enumerate}
\end{proof}

\begin{ppsn}\label{Image of omega for odd}
Let $\pi$ be a representation of $\mathcal{B}_{2n+1}$. Then:
\begin{enumerate}[label=(\arabic*)]
    \item\label{IO.Part1} The operator $\pi(\omega)$ is identity precisely on the subspace $\bigcap_{i=1}^{2n} \ker \pi\left(z_i^*\right)$.
    \item The following kernel inclusions are satisfied:
    \begin{itemize}
        \item For $i \in \{1, \dots, 2n+1\} \setminus \{n+1\}$, we have $\ker\pi\left(z_i\right) \subseteq \ker\pi\left(z_k^*\right)$ for all $k \geq i$.
        \item For $i = n+1$, we have $\ker\pi\left(z_i\right) \subseteq \ker\pi\left(z_k^*\right)$ for all $k > n+1$.
    \end{itemize}
    \item The operator $\pi\left(z_1\right)$ is an isometry. For $i \in \{2, \dots, 2n+1\}$, the restriction of $\pi\left(z_i\right)$ to a corresponding joint kernel is also an isometry:
    \begin{itemize}\label{IO.Part3}
        \item If $i \neq n+1$, the restriction $\pi\left(z_i\right)|_{\bigcap_{k=1}^{i-1} \ker\pi\left(z_k^*\right)}$ is an isometry.
        \item If $i = n+1$, the restriction $\pi\left(z_i\right)|_{\bigcap_{k=1}^i \ker\pi\left(z_k^*\right)}$ is an isometry.
    \end{itemize}
    \item\label{IO.Part4} Either $\sigma\left(\pi (\omega)\right) = \{0,1\}$ or $\sigma\left(\pi(\omega)\right) =\{0\}$.
\end{enumerate}
\end{ppsn}

\begin{proof}
\begin{enumerate}[label=(\arabic*)]
    \item If $h \in \bigcap_{i=1}^{2n} \ker \pi\left(z_i^*\right)$, then relation \eqref{r5.0} shows that
    $$
    \pi(\omega)h = \left(I - \sum_{i=1}^{2n} \pi\left(z_i\right)\pi\left(z_i^*\right)\right)h = h.
    $$
    Conversely, if $\pi(\omega)h = h$, the relation $z_i^*\omega = 0$ from Lemma \ref{Properties of omega for odd} implies that for each $i \in \{1, \dots, 2n\}$, $\pi\left(z_i^*\right)h = \pi\left(z_i^*\right)\pi(\omega)h = \pi\left(z_i^*\omega\right)h = 0$. Therefore, $h \in \bigcap_{i=1}^{2n} \ker \pi\left(z_i^*\right)$.
    
    \item First, consider the case where $i \in \{1, \dots, 2n+1\} \setminus \{n+1\}$. Let $h \in \ker\pi\left(z_i\right)$. Applying the representation $\pi$ to relation~\eqref{r3.0} and evaluating at $h$ gives
$$
\left(\pi\left(z_i z_i^*\right) + \sum_{k>i} \pi\left(z_k z_k^*\right)\right)h = \pi\left(z_i^*z_i\right)h = 0.
$$
Taking the inner product of the left-hand side with $h$ yields a sum of squared norms:
$$
\left\|\pi\left(z_i^*\right)h\right\|^2 + \sum_{k>i} \left\|\pi\left(z_k^*\right)h\right\|^2 = 0.
$$
Thus $\pi\left(z_k^*\right)h=0$ for every $k\geq i$,
proving the first inclusion.

For the case $i=n+1$, the inclusion follows by a similar argument using relation \eqref{r4.0}.

   \item For the case $i=1$, relation~\eqref{r3.0} states that $z_1^*z_1 = \sum_{k=1}^{2n+1} z_k z_k^*$. By relation~\eqref{r5.0}, this sum is equal to $1$. Applying the representation $\pi$ yields $\pi(z_1)^*\pi(z_1) = I$, which proves that $\pi(z_1)$ is an isometry.

Next, consider $i \in \{2, \dots, 2n+1\}$ with $i \neq n+1$. Let $W_{i-1} := \bigcap_{k=1}^{i-1} \ker\pi\left(z_k^*\right)$ and let $h \in W_{i-1}$. Since $\pi(z_k^*)h = 0$ for all $k < i$, we have $\pi\left(\sum_{k=1}^{i-1} z_k z_k^*\right)h = 0$. We can then compute:
\begin{align*}
    \pi\left(z_i\right)^*\pi\left(z_i\right)h
    &= \pi\left(\sum_{k=i}^{2n+1} z_k z_k^*\right)h && \left(\text{by relation \eqref{r3.0}}\right) \\
    &= \left( \pi\left(\sum_{k=1}^{i-1} z_k z_k^*\right) + \pi\left(\sum_{k=i}^{2n+1} z_k z_k^*\right) \right) h && \left(\text{since } \pi\left(\sum_{k=1}^{i-1} z_k z_k^*\right)h = 0\right) \\
    &= \pi\left(\sum_{k=1}^{2n+1} z_k z_k^*\right)h = h && \left(\text{by relation \eqref{r5.0}}\right).
\end{align*}
This shows that the restriction $\pi\left(z_i\right)|_{W_{i-1}}$ is an isometry.

Finally, the proof for the case $i=n+1$ is analogous, using relation \eqref{r4.0} in place of relation \eqref{r3.0}.

    \item Since $\omega$ is a projection, $\sigma\left(\pi(\omega)\right) \subseteq \{0, 1\}$. Assume for contradiction that $\sigma\left(\pi(\omega)\right) = \{1\}$. This would imply $\pi(\omega) = I$. From Lemma \ref{Properties of omega for odd}, the relation $\omega z_i = 0$ gives $\pi\left(z_i\right) = \pi(\omega)\pi\left(z_i\right) = 0$ for all $i < 2n+1$. This yields a contradiction via relation \eqref{r4.0}:
    $$
    I = \pi(\omega) = \pi\left(z_{n+1}^* z_{n+1}\right) = \pi\left(z_{n+1}\right)^*\pi\left(z_{n+1}\right) = 0.
    $$
    The assumption that $\sigma\left(\pi(\omega)\right) = \{1\}$ is therefore false, which establishes the claim.
\end{enumerate}
\end{proof}

\begin{ppsn}\label{Invariance for odd}
Let $u\in\bigcap_{i=1}^{2n} \ker \pi\left(z_i^*\right)$. Then 
$$
\pi\left(z_{2n+1}\right)u \in \bigcap_{i=1}^{2n} \ker \pi\left(z_i^*\right), \text{ and} \quad \pi\left(z_{2n+1}^*\right)u \in \bigcap_{i=1}^{2n} \ker \pi\left(z_i^*\right).
$$
\end{ppsn}

\begin{proof}
Fix $i \in \{1,2,\ldots,2n\}$. By relation \eqref{r2.0}, $\pi\left(z_i^*\right)\pi\left(z_{2n+1}\right)u = 0$, and by relation \eqref{r1.0}, $\pi\left(z_i^*\right)\pi\left(z_{2n+1}^*\right)u = 0$. Therefore both $\pi\left(z_{2n+1}\right)u$ and $\pi\left(z_{2n+1}^*\right)u$ belong to $\bigcap_{i=1}^{2n} \ker \pi\left(z_i^*\right)$, as claimed.
\end{proof}

Let $K \subseteq \bigcap_{i=1}^{2n} \ker\pi(z_i^*)$ be a subspace invariant under the action of the $C^*$-algebra generated by $\pi(z_{2n+1})$. Define the index set
$$
\mathcal{I}_{2n} := \left\{\alpha = \left(\alpha_1, \dots, \alpha_{2n}\right) \in\mathbb{N}_0^{2n} : \alpha_{n+1}\leq 1\right\}.
$$
For each $u \in K$ and $\alpha \in \mathbb{N}_0^{2n}$, define the vector $u_\alpha$ by
$$
u_\alpha := \pi\left(z_1\right)^{\alpha_1}\pi\left(z_2\right)^{\alpha_2}\cdots \pi\left(z_{2n}\right)^{\alpha_{2n}}u.
$$

\begin{lmma}\label{Basis vectors norm for odd}
Let $\pi$ be a representation of $\mathcal{B}_{2n+1}$. Then, for each
$u \in \bigcap_{i=1}^{2n} \ker \pi\left(z_i^*\right)$ and
$\alpha \in \mathcal{I}_{2n}$, we have $\left\| u_{\alpha} \right\| = \|u\|$.
\end{lmma}

\begin{proof}
Fix $u \in \bigcap_{i=1}^{2n} \ker \pi\left(z_i^*\right)$ and
$\alpha \in \mathcal{I}_{2n}$. We first claim that for each $1 \leq i < 2n$,
$$\left\|
\pi\left(z_i\right)^{\alpha_i}
\pi\left(z_{i+1}\right)^{\alpha_{i+1}}
\cdots
\pi\left(z_{2n}\right)^{\alpha_{2n}}u
\right\|
=
\left\|
\pi\left(z_{i+1}\right)^{\alpha_{i+1}}
\cdots
\pi\left(z_{2n}\right)^{\alpha_{2n}}u
\right\|.$$

If $i \neq n+1$, then for any $\beta_i \in \mathbb{N}_0$, relation \eqref{r2.0} implies 
$$\pi\left(z_i\right)^{\beta_i}
\pi\left(z_{i+1}\right)^{\alpha_{i+1}}
\cdots
\pi\left(z_{2n}\right)^{\alpha_{2n}}u
\in
\bigcap_{k=1}^{i-1} \ker \pi\left(z_k^*\right).$$ Hence the claim follows by repeated application of Proposition \ref{Image of omega for odd}\ref{IO.Part3}.

If $i=n+1$, then $\alpha_i \in \{0,1\}$. The case $\alpha_i=0$ is immediate. If $\alpha_i=1$, then relation \eqref{r2.0} gives 
$$\pi\left(z_{i+1}\right)^{\alpha_{i+1}}
\cdots
\pi\left(z_{2n}\right)^{\alpha_{2n}}u
\in
\bigcap_{k=1}^{i} \ker \pi\left(z_k^*\right),$$
and the claim again follows from Proposition \ref{Image of omega for odd}\ref{IO.Part3}.

Iterating the above claimed equality yields $\|u_\alpha\|=\left\|\pi\left(z_{2n}\right)^{\alpha_{2n}}u\right\|$.

If $n=1$, then $\alpha_{2n}\in\{0,1\}$, and Proposition \ref{Image of omega for odd}\ref{IO.Part3} implies
$\left\|\pi\left(z_{2n}\right)^{\alpha_{2n}}u\right\|=\|u\|$.

If $n>1$, then relation \eqref{r2.0} shows that
$\pi\left(z_{2n}\right)^{\beta_{2n}}u \in \bigcap_{k=1}^{2n-1} \ker \pi\left(z_k^*\right)$ for all $\beta_{2n} \in \mathbb{N}_0$, and the same proposition again gives $\left\|\pi\left(z_{2n}\right)^{\alpha_{2n}}u\right\|=\|u\|$.

Thus $\|u_\alpha\|=\|u\|$, completing the proof.
\end{proof}

For each $i \in \left\{1,2,\ldots,2n\right\}$, we denote by $\varepsilon_i \in \mathbb{N}_0^{2n}$ the multi-index whose $i$-th entry is $1$ and whose remaining entries are $0$.

\begin{lmma}\label{Representation on generating elements for odd}
Let $\pi$ be a representation of $\mathcal{B}_{2n+1}$ and let $K := \bigcap_{i=1}^{2n} \ker \pi\left(z_i^*\right)$. For any $u \in K$ and any multi-index $\alpha = \left(\alpha_1, \dots, \alpha_{2n}\right) \in \mathbb{N}_0^{2n}$, the following relations hold:
\begin{enumerate}[label=(\arabic*)]
    \item The action of $\pi\left(z_{2n+1}\right)$ is given by
    $$ \pi\left(z_{2n+1}\right) u_{\alpha} = \begin{cases}
    \pi\left(z_{2n+1}\right)u & \text{if } \alpha = 0, \\
    0 & \text{otherwise}.
    \end{cases} $$

    \item The action of $\pi\left(z_{2n+1}^*\right)$ is given by
    $$ \pi\left(z_{2n+1}^*\right)u_{\alpha} = \begin{cases}
    \pi\left(z_{2n+1}^*\right)u & \text{if } \alpha = 0, \\
    0 & \text{otherwise}.
    \end{cases} $$

    \item For $i \in \{1, \dots, 2n\}$, the action of $\pi\left(z_i\right)$ is given by
    $$ \pi\left(z_i\right)u_{\alpha} = \begin{cases}
    u_{\alpha+{\varepsilon}_i} & \text{if } \alpha_1 = \cdots = \alpha_{i-1}=0, \\
    0 & \text{otherwise}.
    \end{cases}$$

    \item For $i \in \{1, \dots, 2n\}$, the action of $\pi\left(z_i^*\right)$ is given by
    $$\pi\left(z_i^*\right) u_{\alpha} = \begin{cases}
    u_{\alpha - \varepsilon_i} & \text{if } \alpha\in\mathcal{I}_{2n},\,\alpha_1 = \cdots = \alpha_{i-1} =0, \text{ and } \alpha_i\geq 1,\\
    0 & \text{otherwise}.
    \end{cases}$$
\end{enumerate}
\end{lmma}

\begin{proof}
The assertions (1) and (3) follow directly from relation \eqref{r1.0}. The assertion in (2) is an immediate consequence of relation \eqref{r2.0}. Finally, (4) follows from relation \eqref{r2.0} together with Proposition \ref{Image of omega for odd}\ref{IO.Part3}.
\end{proof}

\begin{lmma}\label{Joint kernel nontriviality for odd}
Let $\pi$ be a representation of $\mathcal{B}_{2n+1}$. If $\pi\left(z_{2n+1}\right) \neq 0$, then the space $\bigcap_{i=1}^{2n} \ker\pi\left(z_i^*\right)$ is non-zero.
\end{lmma}

\begin{proof}
Since $\pi\left(z_{2n+1}\right)\neq0$, we have
$$
\pi(\omega)
=
\pi\left(z_{2n+1}\right)^*
\pi\left(z_{2n+1}\right)
\neq0.
$$
By Proposition \ref{Image of omega for odd}(1),
$\pi(\omega)$ is the projection onto
$\bigcap_{i=1}^{2n}\ker\pi\left(z_i^*\right)$.
Hence this subspace is nonzero.
\end{proof}

\begin{ppsn}\label{Joint kernel dimension for odd}
If $\pi$ is an irreducible representation of $\mathcal{B}_{2n+1}$ such that $\pi\left(z_{2n+1}\right) \neq 0$, then the space $\bigcap_{i=1}^{2n} \ker\pi\left(z_i^*\right)$ is one-dimensional.
\end{ppsn}

\begin{proof}
Let
$$
K:=\bigcap_{i=1}^{2n}\ker\pi\left(z_i^*\right).
$$
By Lemma \ref{Joint kernel nontriviality for odd},
$K\neq\{0\}$.

Suppose that $\dim(K)>1$. Consider the action of the commutative C*-algebra $C^*\left(z_{2n+1}\right)$ generated by the element $z_{2n+1}$ on the subspace $K$. Since $\dim\left(K\right)>1$, there exists a proper non-zero closed subspace $W \subsetneq K$ that is invariant under the action of this algebra. Let $H^W$ denote the closed linear span of all vectors $u_\alpha$ with $u\in W$ and $\alpha\in\mathcal{I}_{2n}$; that is,
$$
H^W := \overline{\operatorname{span}}\left\{u_\alpha : u\in W,\ \alpha\in\mathcal{I}_{2n}\right\}.
$$
This subspace is nonzero. By Lemma
\ref{Representation on generating elements for odd},
it is invariant under each $\pi(z_i)$ and $\pi(z_i^*)$,
and hence under $\pi$.

By Lemma \ref{Properties of omega for odd}(1) and
Proposition \ref{Image of omega for odd}(1),
$$
\pi(\omega)u_\alpha
=
\begin{cases}
u,&\text{if }\alpha=0,\\
0,&\text{otherwise}.
\end{cases}
$$
Since $W$ is closed, it follows that
$\pi(\omega)(H^W)=W$. On the other hand,
$\pi(\omega)(H)=K$. Thus $H^W\neq H$, since $W\neq K$.

We have obtained a proper nonzero closed
$\pi$-invariant subspace of $H$, contradicting
irreducibility. Therefore $\dim(K)=1$.
\end{proof}

\begin{lmma}\label{ON basis for odd}
Let $\pi$ be an irreducible representation of $\mathcal{B}_{2n+1}$ on a Hilbert space $H$ such that $\pi\left(z_{2n+1}\right)\neq 0$. If $u$ is an element of $\bigcap_{i=1}^{2n}\ker\pi\left(z_i^*\right)$ such that $\left\|u\right\| = 1$, then the set $\left\{u_{\alpha}: \alpha \in \mathcal{I}_{2n}\right\}$ is an orthonormal basis for $H$.
\end{lmma}

\begin{proof}
Denote $K:=\bigcap_{i=1}^{2n}\ker\pi\left(z_i^*\right)$. By Lemma \ref{Representation on generating elements for odd}, the subspace $H^K$ is non-zero and $\pi$-invariant. Since $\pi$ is irreducible, we have $H^K=H$. Consequently, $\overline{\text{span}}\left\{u_{\alpha}:\alpha\in\mathcal{I}_{2n}\right\}=H$.

It remains to show that $\left\{u_{\alpha}:\alpha\in\mathcal{I}_{2n}\right\}$ is orthogonal. Let $\alpha,\beta\in\mathcal{I}_{2n}$ and assume that $\left\langle u_{\alpha},u_{\beta}\right\rangle\neq 0$. For $1\leq i\leq 2n$, denote
$$
u_{\alpha}^{(i+1)}=
\begin{cases}
\pi\left(z_{i+1}\right)^{\alpha_{i+1}}\cdots\pi\left(z_{2n}\right)^{\alpha_{2n}}u,& 1\leq i\leq 2n-1,\\
u,& i=2n,
\end{cases}
$$
and denote $u_{\beta}^{(i+1)}$ analogously. We show by induction on $i\in\left\{1,\ldots,2n\right\}$ that $\alpha_i=\beta_i$ and $\left\langle u_{\alpha},u_{\beta}\right\rangle
=
\left\langle u_{\alpha}^{(i+1)},u_{\beta}^{(i+1)}\right\rangle$.

For $i=1$, suppose $\alpha_1<\beta_1$. Since $\pi\left(z_1\right)$ is an isometry by Proposition \ref{Image of omega for odd}(3), we have $\pi\left(z_1^*\right)\pi\left(z_1\right)=I$. Then
\begin{align*}
\left\langle u_{\alpha},u_{\beta}\right\rangle
&=\left\langle u_{\alpha}^{(2)},\pi\left(z_1^*\right)^{\alpha_1}\pi\left(z_1\right)^{\beta_1}u_{\beta}^{(2)}\right\rangle\\
&=\left\langle \pi\left(z_1^*\right)\pi\left(z_2\right)^{\alpha_2}\cdots\pi\left(z_{2n}\right)^{\alpha_{2n}}u,
\pi\left(z_1\right)^{\beta_1-\alpha_1-1}u_{\beta}^{(2)}\right\rangle.
\end{align*}
By relation \eqref{r2.0}, the first vector in the last inner product is zero, contradicting our assumption. A similar argument applies if $\beta_1<\alpha_1$. Hence $\alpha_1=\beta_1$, and $\left\langle u_{\alpha},u_{\beta}\right\rangle
=
\left\langle u_{\alpha}^{(2)},u_{\beta}^{(2)}\right\rangle$.

Assume now that for some $1\leq j<2n$ we have $\alpha_j=\beta_j$ and $\left\langle u_{\alpha},u_{\beta}\right\rangle
=
\left\langle u_{\alpha}^{(j+1)},u_{\beta}^{(j+1)}\right\rangle$.

Suppose $\alpha_{j+1}<\beta_{j+1}$. Then
$$
\left\langle u_{\alpha}^{(j+1)},u_{\beta}^{(j+1)}\right\rangle
=
\left\langle u_{\alpha}^{(j+2)},
\pi\left(z_{j+1}^*\right)^{\alpha_{j+1}}\pi\left(z_{j+1}\right)^{\beta_{j+1}}u_{\beta}^{(j+2)}\right\rangle.
$$
If $j\neq n$, then by relation \eqref{r2.0} we have, for every $\gamma_{j+1}\in\mathbb{N}_0$, that
$$
\pi\left(z_{j+1}\right)^{\gamma_{j+1}}u_{\beta}^{(j+2)}
\in
\bigcap_{k=1}^j \ker\pi\left(z_k^*\right).
$$
Consequently, repeated application of Proposition \ref{Image of omega for odd}\ref{IO.Part3} yields
$$
\pi\left(z_{j+1}^*\right)^{\alpha_{j+1}}\pi\left(z_{j+1}\right)^{\beta_{j+1}}u_{\beta}^{(j+2)}
=
\pi\left(z_{j+1}\right)^{\beta_{j+1}-\alpha_{j+1}}u_{\beta}^{(j+2)}.
$$
If $j=n$, then $\alpha_{j+1}=0$ and $\beta_{j+1}=1$, and the same identity holds. Thus,
$$
\left\langle u_{\alpha}^{(j+1)},u_{\beta}^{(j+1)}\right\rangle
=
\left\langle \pi\left(z_{j+1}^*\right)u_{\alpha}^{(j+2)},
\pi\left(z_{j+1}\right)^{\beta_{j+1}-\alpha_{j+1}-1}u_{\beta}^{(j+2)}\right\rangle.
$$
By relation \eqref{r2.0}, the first vector in the above inner product is zero. Hence, $\left\langle u_{\alpha}^{(j+1)},u_{\beta}^{(j+1)}\right\rangle=0$, which implies $\left\langle u_{\alpha},u_{\beta}\right\rangle=0$. This contradicts our assumption. The case $\beta_{j+1}<\alpha_{j+1}$ is similar. Hence $\alpha_{j+1}=\beta_{j+1}$, and $\left\langle u_{\alpha}^{(j+1)},u_{\beta}^{(j+1)}\right\rangle
=
\left\langle u_{\alpha}^{(j+2)},u_{\beta}^{(j+2)}\right\rangle$. By induction, $\alpha=\beta$. Therefore the family $\left\{u_{\alpha}:\alpha\in\mathcal{I}_{2n}\right\}$ is an orthonormal basis of $H$.
\end{proof}

Let $\pi:\mathcal{B}_{2n+1}\to\mathcal{L}(H)$ be an irreducible
representation. First suppose that $\pi(z_{2n+1})\neq 0$, and set
$$
K:=\bigcap_{i=1}^{2n}\ker\pi\left(z_i^*\right).
$$
We describe $\pi$ in the following steps.

\begin{itemize}
\item By Proposition~\ref{Image of omega for odd}\ref{IO.Part4}, $\sigma\left(\pi\left(z_{2n+1}^*z_{2n+1}\right)\right)=\{0,1\}$. Proposition \ref{Image of omega for odd}\ref{IO.Part1} identifies $K$ as the eigenspace corresponding to the eigenvalue $1$. Moreover, by Proposition \ref{Joint kernel dimension for odd},
$K$ is one-dimensional. Fix a unit vector $u\in K$.

\item By Proposition \ref{Invariance for odd}, $K$ is invariant under $\pi\left(z_{2n+1}\right)$. Hence $\pi\left(z_{2n+1}\right)u=tu$ for some $t\in\mathbb{C}$. Since $\pi\left(z_{2n+1}^*z_{2n+1}\right)$ acts as the identity on $K$, it follows that $t\in\mathbb{T}$.

\item By Lemma \ref{ON basis for odd}, $\left\{u_\alpha:\alpha\in\mathcal{I}_{2n}\right\}$ is an orthonormal
basis of $H$. Lemma \ref{Representation on generating elements for odd} describes the action of each $\pi\left(z_i\right)$ on this basis in terms of $t$. Thus $\pi$ is determined, up to unitary equivalence, by $t$. We denote the corresponding representation by
$\pi_{t,2n+1}$.
\end{itemize}

More generally, let $k$ be the largest index such that
$\pi(z_k)\neq 0$. Such an index exists by relation~\eqref{r5.0}.
Relation~\eqref{r4.0} implies that $k\neq n+1$, while
relation~\eqref{r3.0} shows that $\pi(z_k)$ is normal.
Applying the same analysis gives the following two cases.

\medskip
\noindent\textbf{Case 1.} If $k\geq 2$, then $\bigcap_{i=1}^{k-1} \ker \pi\left(z_i^*\right)$ is one-dimensional, and $\pi$ is uniquely determined by a scalar $t\in\mathbb{T}$ such that $\pi\left(z_k\right)u=tu$, where $u\in\bigcap_{i=1}^{k-1}\ker \pi\left(z_i^*\right)$ and $\|u\|=1$. We denote this representation by $\pi_{t,k}$.

\medskip
\noindent\textbf{Case 2.} If $k=1$, then the representing Hilbert space is one-dimensional, and $\pi$ is uniquely determined by a scalar $t\in\mathbb{T}$ such that $\pi\left(z_1\right)u=tu$, where $u\in H\setminus\{0\}$. We denote this representation by $\pi_{t,1}$.

The above characterization of the irreducible representations of $\mathcal{B}_{2n+1}$ leads to the following theorem.

\begin{thm}\label{Characterization of irr. repn odd case}
Let $\pi : \mathcal{B}_{2n+1} \rightarrow \mathcal{L}\left(H\right)$ be an irreducible representation.
Then $\pi$ is determined, up to unitary equivalence, by a unique pair
$(t,k) \in \mathbb{T} \times \left\{1,2,\ldots,n,n+2,n+3,\ldots,2n+1\right\}$,
characterized as follows.

\begin{enumerate}[label=\normalfont(\roman*)]
\item
$k \in \left\{1,2,\ldots,2n+1\right\} \setminus \left\{n+1\right\}$ is the unique index such that
$\pi\left(z_k\right) \neq 0$ and $\pi\left(z_i\right)=0$ for all
$k<i\leq 2n+1$.

\item
The scalar $t$ is uniquely determined as follows.
\begin{itemize}
\item
If $k=1$, then $\pi\left(z_1\right)u=t u$ for all $u \in H$.
\item
If $k \geq 2$, then the subspace
$\bigcap_{i=1}^{k-1} \ker \pi\left(z_i^*\right)$ is one-dimensional, and
$\pi\left(z_k\right)u=t u$ for all
$u \in \bigcap_{i=1}^{k-1} \ker \pi\left(z_i^*\right)$.
\end{itemize}
\end{enumerate}
\end{thm}

Since $\mathcal{B}_{2n+1}$ is universal, there exists a surjective $^*$-homomorphism
$$\iota : \mathcal{B}_{2n+1} \rightarrow C\left(SO_0\left(2n+1\right)/SO_0\left(2n-1\right)\right)$$
sending the generator
$z_i\in\mathcal{B}_{2n+1}$ to the element $z_i\in C\left(SO_0\left(2n+1\right)/SO_0\left(2n-1\right)\right)$.

\brmrk \label{Existence of irr. repn for each tuple}
By Theorem \ref{Characterization of irr. repn odd case}, each irreducible representation of $\mathcal{B}_{2n+1}$ is determined, up to unitary equivalence, by a unique pair
$(t,k)\in\mathbb{T}\times\left\{1,2,\ldots,n,n+2,n+3,\ldots,2n+1\right\}$.
As in the preceding discussion, we denote the corresponding representation by $\pi_{t,k}$.

We now show that every such pair occurs. For each $(t,k)$, define
$$
\pi_{t,k}=\begin{cases}
\sigma_{(-1)^{k-1}t,k-1}\circ\iota, &\text{if } 1\leq k\leq n,\\
\sigma_{(-1)^{n-1}t,k-2}\circ\iota, &\text{if } n+2\leq k\leq 2n+1.
\end{cases}
$$
Since $\iota$ is surjective, Proposition \ref{Natural irr. repn} implies that each $\pi_{t,k}$ is irreducible. Moreover, Remark \ref{Indexing of the irreducible representations} shows that its parameters in Theorem \ref{Characterization of irr. repn odd case} are precisely $(t,k)$.

Thus every pair $(t,k)$ corresponds to an irreducible representation of $\mathcal{B}_{2n+1}$, unique up to unitary equivalence. From now on, we use $\pi_{t,k}$ for the representative defined above.
\ermrk

With the notation introduced in Remark \ref{Existence of irr. repn for each tuple}, the following corollary is an immediate consequence of Theorem \ref{Characterization of irr. repn odd case}.

\begin{crlre}\label{All irr. repn odd}
Up to unitary equivalence, the family
$$\left\{\pi_{t,k} : t\in\mathbb{T},\,k \in \left\{1,2,\ldots,n,n+2,n+3,\ldots,2n+1\right\}\right\}$$
is a complete set of irreducible representations of $\mathcal{B}_{2n+1}$.
\end{crlre}

\begin{thm}\label{Universality of Crystallization}
The $^*$-homomorphism $\iota : \mathcal{B}_{2n+1} \rightarrow C\left(SO_0\left(2n+1\right)/SO_0\left(2n-1\right)\right)$ is an isomorphism.
\end{thm}

\begin{proof}
Since $\iota$ is surjective, it remains to prove that $\iota$ is injective. For this, it suffices to show that every irreducible representation of $\mathcal{B}_{2n+1}$ factors through $\iota$. This follows from Corollary \ref{All irr. repn odd} and the definition of $\pi_{t,k}$ in Remark \ref{Existence of irr. repn for each tuple}.
\end{proof}

\begin{crlre} \label{Characterization of all irr. repn odd}
Up to unitary equivalence, the family
$$
\left\{\sigma_{t,m}: t\in\mathbb{T},\, m\in\left\{0,1,\ldots, 2n-1\right\}\right\}
$$
forms a complete set of irreducible representations of $C\left(SO_0(2n+1)/SO_0(2n-1)\right)$.
\end{crlre}

\begin{proof}
The proof follows directly from Corollary~\ref{All irr. repn odd}, and Theorem~\ref{Universality of Crystallization}.
\end{proof}

Using the representations defined in Remark \ref{Existence of irr. repn for each tuple}, we obtain the following classification of the irreducible representations of the quotient $C^*$-algebras $\mathcal{B}_{2n+1}^{k}$.

\begin{ppsn} \label{Irr. repn for filtration}
For each $1\leq k\leq 2n+1$ with $k\neq n+1$, the family
$$
\left\{\pi_{t,j}: t\in\mathbb{T},\ 1\leq j\leq k,\ j\neq n+1\right\}
$$
is a complete set of irreducible representations of $\mathcal{B}_{2n+1}^{k}$, up to unitary equivalence. Here, we use the same notation for $\pi_{t,j}$ and
the representation it induces on the quotient.
\end{ppsn}

\begin{proof}
Via the quotient map, the irreducible representations of
$\mathcal{B}_{2n+1}^{k}$ correspond exactly to those of
$\mathcal{B}_{2n+1}$ that vanish on $z_i$ for every $i>k$.
By Theorem~\ref{Characterization of irr. repn odd case},
these are precisely the representations $\pi_{t,j}$ with
$j\leq k$.
\end{proof}

Let $0\leq m\leq 2n-1$. Define the $C^*$-homomorphism
$$
\sigma_m : C\left(SO_0\left(2n+1\right)/SO_0\left(2n-1\right)\right)
\to C\left(\mathbb{T}\right)\otimes \mathscr{T}^{\otimes m}
$$
by
$$\sigma_m(a)(t)=\sigma_{t,m}(a), \qquad  
a\in C\left(SO_0\left(2n+1\right)/SO_0\left(2n-1\right)\right).$$

Proceeding as in the proof of \cite[Corollary 3.8]{BhuBisSau-2026aa}, it follows that $\sigma_{2n-1}$ is a faithful representation of $C\left(SO_0\left(2n+1\right)/SO_0\left(2n-1\right)\right)$. For $0 \leq m \leq 2n-1$, we define
$$
\ell(m) =
\begin{cases}
m+1, & \text{if } 0\leq m\leq n-1,\\
m+2, & \text{if } n\leq m\leq 2n-1.
\end{cases}
$$

\begin{ppsn}
Let $0 \leq m \leq 2n-2$. Then the $C^*$-algebra $\sigma_m\left(C\left(SO_0\left(2n+1\right)/SO_0\left(2n-1\right)\right)\right)$ is isomorphic to $\mathcal{B}^{\ell(m)}_{2n+1}$.
\end{ppsn}

\begin{proof}
By the defining relations, there is a natural surjective
$^*$-homomorphism
$$
\iota_m:\mathcal{B}_{2n+1}^{\ell(m)}
\longrightarrow
\sigma_m\left(
C\left(SO_0\left(2n+1\right)/SO_0\left(2n-1\right)\right)
\right)
$$
such that $\iota_m(z_i)=\sigma_m(z_i)$ for
$1\leq i\leq2n+1$.

Let $a\in\ker\iota_m$. Evaluation at any $t\in\mathbb{T}$
gives
$$
\pi_{t,\ell(m)}(a)
=
\operatorname{ev}_t\left(\iota_m(a)\right)
=
0.
$$
Since $\pi_{t,j}$ factors through $\pi_{t,\ell(m)}$ and
$\pi_{t,\ell(m)}(a)=0$, we have $\pi_{t,j}(a)=0$ for every
$t\in\mathbb{T}$ and every $1\leq j<\ell(m)$ with $j\neq n+1$. Proposition \ref{Irr. repn for filtration} therefore shows
that every irreducible representation of
$\mathcal{B}_{2n+1}^{\ell(m)}$ annihilates $a$.
Hence $a=0$.

Thus $\iota_m$ is injective and consequently an isomorphism.
\end{proof}

\brmrk \label{matrix unit}
Let $\pi : C\left(SO_0\left(2n+1\right)/SO_0\left(2n-1\right)\right)\longrightarrow \mathcal{L}\left(H\right)$ be an irreducible representation. Recall the isomorphism $\iota$ described in Theorem~\ref{Universality of Crystallization}. Then, by Corollary~\ref{All irr. repn odd}, the representation $\pi\circ\iota$ is unitarily equivalent to $\pi_{t,k}$ for some $t\in\mathbb{T}$ and some $k\in\left\{1,2,\ldots,2n+1\right\}$ with $k\neq n+1$. Moreover, $\bigcap_{i=1}^{k-1} \ker \pi\left(z_i^*\right)$ (with the notational convention that $\bigcap_{i=1}^{k-1} \ker\pi\left(z_i^*\right) := H$ if $k=1$) is one-dimensional. If $k\geq 2$, then, by Lemma \ref{ON basis for odd}, by considering suitable monomials in the generators $z_1,z_2,\ldots,z_{k-1}$ and applying their images under $\pi$ to any fixed nonzero vector $u\in \bigcap_{i=1}^{k-1}\ker \pi\left(z_i^*\right)$, one obtains a basis of the representing Hilbert space $H$.
\ermrk

\brmrk
Whenever it is necessary to emphasize that the irreducible representation $\pi_{t,j}$, where $t\in\mathbb{T}$ and $1\leq j\leq k$ with $j\neq n+1$, is regarded as a representation of $\mathcal{B}_{2n+1}^{k}$, we shall denote it by $\pi_{t,j}^{k,n}$.
\ermrk

\subsection{Primitive ideal space}

For each $t \in \bbbt$ and $1 \leq r \leq k$, define  $I_{t,r}^{k,n} = \ker \pi_{t,r}^{k,n}$.
The primitive ideal space of $\mathcal{B}_{2n+1}^k$ is given by the set 
\[
\text{Prim}\left(\mathcal{B}_{2n+1}^k\right) = \left\{ I_{t,r}^{k,n} \;\middle|\; t \in T, \, 1 \le r \le k \right\}.
\]
equipped with the hull kernel toplogy $\tau$.
We define a map $\Psi: \bbbt \times \{1, 2, \dots, k\} \to \text{Prim}\left(\mathcal{B}_{2n+1}^k\right)$ by 
\[
\Psi(t, r) = \ker \pi_{t,r}^{k,n}.
\]

Let $A \subset \bbbt \times \{1, 2, \dots, k\}$. For each $r \in \{1, 2, \dots, k\}$, let $A_r$ denote the slice of $A$ at level $r$, defined as:
\[
A_r = A \cap (\bbbt \times \{r\}).
\]
Clearly, $A$ can be decomposed as the disjoint union $A = A_1 \cup A_2 \cup \dots \cup A_k$. We define the index of $A$ to be
\[
\text{ind}_A = \inf \{ r \mid A_r \neq \emptyset \}.
\]
\bdfn
A subset $A \subset \bbbt \times \{1, 2, \dots, k\}$ is said to be \emph{closed} under the topology $\mathcal{O}$ if:
\begin{enumerate}
	\item[\rm (1)] $A_{\text{ind}_A}$ is a closed set in $\bbbt$, and
	\item[\rm (2)] $A_r = \bbbt$ for all $r > \text{ind}_A$.
\end{enumerate}
\edfn

We equip the set  $\bbbt \times \{1, 2, \dots, k\}$ with this topology $\mathcal{O}$.  We will show that the primitive ideal space of $\mathcal{B}_{2n+1}^k$ is homemorphic to $(\bbbt \times \{1, 2, \dots, k\}, \mathcal{O})$.  For that, we need to establish some results.

\bppsn \label{P2}
Let $1\leq s\leq r  \leq  k$ and let $t \in \bbbt$. Then $I_{t,r}^{k,n} \subseteq I_{t,s}^{k,n}$. 
\eppsn 
\prf It follows from the fact that the irreducible representation $\pi_{t,s}^{k,n}$ factors through $\pi_{t,r}^{k,n}$. 
\qed 

\bppsn \label{P3}
For each $t \in \bbbt$ and $1 \leq r \leq k$, the primitive ideal $I_{t,r}^{k,n}$ is generated by the following elements:
\[
I_{t,r}^{k,n} = \langle z_k, z_{k-1}, \dots, z_{r+1}, z_r - t z_r z_r^* \rangle.
\]
\eppsn

\prf
Let $J = \langle z_k, z_{k-1}, \dots, z_{r+1}, z_r - t z_r z_r^* \rangle$. It follows directly from the description of the faithful representation $\pi_{t,r}^{k,n}$ that $J \subset I_{t,r}^{k,n}$.  

To establish equality, assume for contradiction that $J \subsetneq I_{t,r}^{k,n}$. Because $J$ is the intersection of all primitive ideals containing it, there must exist another primitive ideal $I_{t',s}^{k,n}$ such that $(t', s) \neq (t, r)$, $J \subset I_{t',s}^{k,n}$, and $I_{t,r}^{k,n} \nsubseteq I_{t',s}^{k,n}$.
\begin{itemize}
	\item \textbf{Case 1 ($s < r$):} In this case, Proposition~\ref{P2} implies $I_{t,r}^{k,n} \subseteq I_{t,s}^{k,n}$, yielding a contradiction. 
	
	\item \textbf{Case 2 ($s > r$):} If $s > r$, then by definition $z_s \in I_{t',s}^{k,n}$ but $z_{r} \notin I_{t',s}^{k,n}$ unless $z_r$ is in a higher ideal block. Specifically, the lower generators like $z_r$ are not elements of $I_{t',s}^{k,n}$, which directly contradicts our assumption that $J \subset I_{t',s}^{k,n}$. 
	
	\item \textbf{Case 3 ($s = r$):} By definition, we have:
	\begin{equation}\label{eq1}
		z_r - t' z_r z_r^* \in I_{t',r}^{k,n}.
	\end{equation}
	On the other hand, since $J \subset I_{t',r}^{k,n}$, the generator from $J$ satisfies:
	\begin{equation}\label{eq2}
		z_r - t z_r z_r^* \in I_{t',r}^{k,n}.
	\end{equation}
	Multiplying equations \eqref{eq1} and \eqref{eq2} by $t$ and $t'$ respectively, and subtracting the latter from the former yields $(t-t') z_r \in I_{t',r}^{k,n}$. Since $(t',s) \neq (t,r)$ and $s=r$, we must have $t' \neq t$, implying $z_r \in I_{t',r}^{k,n}$. This is a contradiction since $\pi_{t',r}^{k,n}(z_r)\neq 0$.  
\end{itemize}
Since all cases yield a contradiction, our assumption $J \subsetneq I_{t,r}^{k,n}$ is false. Thus, $J = I_{t,r}^{k,n}$.
\qed

\bthm \label{P1}
Equip  $\text{Prim}(\mathcal{B}_{2n+1}^k)$ the hull-kernel topology $\tau$. Then the map 
$$\Psi: (\bbbt \times \{1, 2, \dots, k\}, \mathcal{O}) \longrightarrow  (\text{Prim}(\mathcal{B}_{2n+1}^k), \tau);\, (t,r) \mapsto \ker \pi^{k,n}_{t,r}$$ 
is a homeomorphism. 
\ethm 
\prf The proof  follows immediately  from  Proposition (\ref{P2}, \ref{P3}).
\qed 

\bcrlre  \label{cor1} Let $1 \leq r \leq k$. Then we have 
\begin{enumerate}
	\item[(i)] $$\cap_{t \in 1\leq r \leq s} \, I_{t,r}^{k,n}= \langle z_k, z_{k-1},\cdots z_s \rangle.$$
	\item[(ii)] If $I$ is an ideal such that $z_r \notin I$, then $I\subset  \langle z_k, z_{k-1},\cdots z_r \rangle.$
\end{enumerate}
\ecrlre

\section{$K$-groups of $C\left(SO_0(2n+1)/SO_0(2n-1)\right)$}
In this section, we first give a faithful representations of $C\left(SO_0(2n+1)/SO_0(2n-1)\right)$, and using that, we compute $K$-groups of $C\left(SO_0(2n+1)/SO_0(2n-1)\right)$ with explicit generators. 
\subsection{Faithful representation of $C\left(SO_0(2n+1)/SO_0(2n-1)\right)$}

Let 
$$
L_{0}^k= \begin{cases}
	\ell^2(\mathbb{N})^{\otimes (k-1)}& \text{ if } 1\leq k \leq n, \cr \\
	\ell^2(\mathbb{N})^{\otimes (n-1)}& \text{ if } k = n+1, \cr \\
	\ell^2(\mathbb{N})^{\otimes (k-2)}\oplus  \ell^2(\mathbb{N})^{\otimes (k-2)} & \text{ if } n+2\leq k \leq 2n.\cr
\end{cases}
$$ 
and $H_0^k=L_0^k\otimes \ell^2(\mathbb{Z})$. 
For $j,k\in \bbn_0$, define  
$$T_j^k=p^{\otimes j}\otimes S^*\otimes 1^{\otimes k}.$$  
Define a map $\Phi^k_{2n+1}$ on the standard generators of $\mathcal{B}^k_{2n+1}$ by     
\[ 
\Phi^k_{2n+1}(z_j)=
T_{j-1}^{k-j}  \quad   \text{ if } 1 \leq j \leq k,\\
\] if $1 \leq k \leq n$. If $k=n+1$, then define 
\[ 
\Phi^k_{2n+1}(z_j)=\begin{cases}
	T_{j-1}^{k-j}  &  \text{ if } 1 \leq j \leq n,\cr \\
	0 &   \text{ if }  j=n+1. \cr 
\end{cases}
\] 
For $n+2\leq k \leq 2n+1$, define 
\[
\Phi^k_{2n+1}(z_j)=\begin{cases}
	T_{j-1}^{k-j-1}   \oplus T_{j-1}^{k-j-1}  & \text{ if } 1 \leq j \leq n,\\
	\begin{pmatrix}
		0 & 0\\
		p^{\otimes n} \otimes 1^{\otimes (k-n-1)} & 0\\
	\end{pmatrix}  & \text{ if } j=n+1,\\
	T_{j-1}^{k-j-1}   \oplus 0 & \text{ if } n+2\leq j \leq k.
\end{cases}
\] 
One can verify that these operators satisfy the defining relations of $\mathcal{B}^k_{2n+1}$. Hence, by the universal property of $\mathcal{B}^k_{2n+1}$, the map $\Phi^k_{2n+1}$ extends to a $*$-homomorphism
\[
\Phi^k_{2n+1}:  \mathcal{B}^k_{2n+1}  \longrightarrow \mathcal{L}(H_0^k),
\]
and therefore defines a representation of $\mathcal{B}_{2n+1}^k$. 

\blmma \label{Ar} Define $A^r=\Phi^r_{2n+1}(\mathcal{B}_{2n+1}^r)$. Then there exists a surjective $C^*$-homomorphism $\Psi^r: A^r\rightarrow A^{r-1}$ mapping $\Phi^r_{2n+1}(z_j) \mapsto \Phi^{r-1}_{2n+1}(z_j) $ for $1\leq j \leq r-1$, and $\Phi^r_{2n+1}(z_r) \mapsto 0$. \elmma 

\prf Let $n+2\leq r \leq 2n+1$. One can consider $A^r$ as a $C^*$-subalgebra of $M_2(\scrt^{\otimes (r-2)})\otimes C(\bbbt)$ by identifying the shift $S^*$ on $\ell^2(\bbz)$ with the unitary generator of $C(\bbbt)$. Let $\gamma_r=1^{\otimes ^{r-3}}\otimes \rho \otimes ev_1$. If we define $\Psi^r$ to be the $C^*$-homomorphism $\begin{pmatrix} \gamma_r& \gamma_r \\ \gamma_r& \gamma_r \\ \end{pmatrix}$ restricted to the domain $ A^r$ and with its codomain restricted to its image $ A^{r-1}$, we obtain the claim. For $2\leq r\leq n+1$, the same argument applies, except that in this case $\gamma_r$ suffices instead of $\begin{pmatrix} \gamma_r& \gamma_r \\ \gamma_r& \gamma_r \\ \end{pmatrix}$. \qed

\bthm \label{faithful}
The representation $\Phi^k_{2n+1}:  \mathcal{B}_{2n+1}^k  \longrightarrow \mathcal{L}(H_0^k)$ is faithful. 
\ethm 
\prf    
We split the proof into three cases.\\
\textbf{Case 1: $1\leq k \leq n$. } The defining relations in this case are identical to those of the $C^*$-algebra of the odd-dimensional quantum sphere $C(S_0^{2k-1})$ (see Remark~\ref{rem2}), and the representation $\Phi^k_{2n+1}$ is the standard faithful representation of $C(S_0^{2k-1})$. Hence, the claim follows.\\
\textbf{Case 2: $k= n+1$. } Note that $z_{n+1}^*z_{n+1}=0$ since $z_{j}=0$ for $n+2\leq j \leq 2n+1$. This implies that $z_{n+1}=0$, and hence $\mathcal{B}^{n+1}_{2n+1}=\mathcal{B}^n_{2n+1}$. The rest follows from the previous case.\\
\textbf{Case 3: $n+2\leq k \leq 2n+1$. } Assume the claim holds for $r=1,2,\dots, k-1$. From Lemma~\ref{Ar}, we have the following surjective $C^*$-homomorphisms:
$$ \mathcal{B}^{r}_{2n+1} \xrightarrow{\Phi^r_{2n+1}} A^r \xrightarrow{\Psi^r} A^{r-1}.$$
By the induction hypothesis for $r=k-1$, we have $A^{k-1}\cong \mathcal{B}_{2n+1}^{k-1}$. Therefore, all irreducible representations of $\mathcal{B}_{2n+1}^k$ whose kernels contain $z_k$ induce a representation of $\mathcal{B}_{2n+1}^{k-1}$, and hence factor through $\Psi^k$, which in turn factor through $\Phi^k_{2n+1}$. Moreover, for the other irreducible representations $\pi_{t,k}$ of $\mathcal{B}_{2n+1}^k$, we have the following explicit factorization:
$$\pi_{t,k}:  \mathcal{B}_{2n+1}^k \xrightarrow{\Phi^k_{2n+1}} A^k \hookrightarrow M_2(\mathscr{T}^{\otimes (k-2)}\otimes C(\bbbt))\xrightarrow{\begin{pmatrix}
		ev_t & 	ev_t \\
		ev_t  & 	ev_t \\
\end{pmatrix}} \mathcal{L}(L_0^k).$$ 
This completes the proof.
\qed \\
From now onwards, we will identify the standard generators $z_i$'s of $\mathcal{B}_{2n+1}^k$  with their image under the map $\Phi^k_{2n+1}$. 
\bcrlre \label{ideal} 
Let $1\leq k \leq 2n+1$, and  let $J_k$ be the closed ideal of $\mathcal{B}_{2n+1}^k$ defined as
$$ J_k= \begin{cases}
	\langle z_k\rangle & \mbox{ if } k\neq n+1, \cr 
	\langle z_n \rangle & \mbox{ if } k=n+1.\cr 
\end{cases} $$ 
Consider $\mathcal{K}(L_0^k) \otimes C(\mathbb{T})$ as an ideal of $M_2(\mathcal{T}^{\otimes (k-1)}) \otimes C(\mathbb{T})$. Then, the equality 
\[
\langle z_k \rangle = \mathcal{K}(L_0^k)  \otimes C(\mathbb{T})
\]
holds as subsets of $M_2(\mathcal{T}^{\otimes (k-1)}) \otimes C(\mathbb{T})$.
\ecrlre
\prf   We will prove the claim for $n+2\leq k \leq 2n+1$ as the other cases follows from the fact that $\mathcal{B}_{2n+1}^k$ is canonically isomorphic to $C(S_0^{2k-1})$.  First note that 
$$z_k= \begin{pmatrix}
	p^{\otimes k-1}\otimes S^* &0\\
	0 & 0\\
\end{pmatrix} = \begin{pmatrix} 
	p^{\otimes k-1}&0\\
	0 & 0\\
\end{pmatrix}\otimes S^*  \in \mathcal{K}(L_0^k)  \otimes C(\mathbb{T}).$$  
This proves that $J_k\subset C(\bbbt)\otimes \clk$. Moreover,  this shows that $p^{\otimes k-1}\otimes C(\bbbt) \subset J_k$.  Hence it is enough to show that 
$$p_{0,i_1} \otimes p_{0,i_2} \otimes \cdots p_{0,i_{k-1}} \otimes 1 \in J_k \mbox{ for all } i_1,i_2,\cdots i_{k-1}\in \bbn_0.$$  
Since $z_i=\pi_{1,k}(z_i)\otimes 1 \in M_2(\scrt^{\otimes (k-1)}) \otimes 1$ for $ 1\leq i \leq k-1$, the claim follows from Remark  \ref{matrix unit}. 
\qed 
\begin{lmma} \label{ses}
	For $1\leq k \leq 2n$ with $k \neq n$, one has the following short exact sequence $\chi_k$ of $C^*$-algebras:
	$$\chi_k:\quad 0\longrightarrow \langle z_{k+1}\rangle \xrightarrow{\iota} \mathcal{B}^{k+1}_{2n+1} \xrightarrow{\alpha_{k}^{k+1}} \mathcal{B}^{k}_{2n+1}\longrightarrow 0.$$ 	
	Moreover, 	the $C^*$-extension $\chi_k$ is a unital homogeneous extension of $\mathcal{B}^{k}_{2n+1}$ by 
	$C(\bbbt)\otimes \clk$. 
\end{lmma}
\prf Fix $1\leq k \leq 2n-1$ with $k \neq n$. Clearly, $z_{k+1}\in \ker (\alpha_{k}^{k+1})$, and hence $\langle z_{k+1}\rangle \subset \ker (\alpha_{k}^{k+1})$. To prove the rest, it suffices to show that all irreducible representations of $\mathcal{B}^{k+1}_{2n+1}$ that vanish on $\langle z_{k+1}\rangle$ factor through 
$\alpha_{k}^{k+1}: \mathcal{B}^{k+1}_{2n+1} \rightarrow\mathcal{B}^{k}_{2n+1}$. Let $\pi$ be an irreducible representation of $\mathcal{B}^{k+1}_{2n+1}$ acting on $\clh$ such that $\pi(z_{k+1})=0$. Then the elements $\{\pi(z_i): 1\leq i \leq k+1\}$ satisfy the defining relations $\mathcal{R}_{2n+1}^{k+1}$ of $\mathcal{B}^{k+1}_{2n+1}$. Since $\pi(z_{k+1})=0$, it is easy to see that $\{\pi(z_i): 1\leq i \leq k\}$ satisfy the defining relations $\mathcal{R}_{2n+1}^k$ of $\mathcal{B}^{k}_{2n+1}$, and hence there exists a $C^*$-homomorphism 
$$\varphi: \mathcal{B}^{k}_{2n+1} \rightarrow \mathcal{L}(\clh): z_i\mapsto \pi(z_i) \quad \mbox{ for } 1\leq i \leq k.$$
Thus, we have the commutative diagram 
\[
\begin{tikzcd}
	\mathcal{B}^{k+1}_{2n+1} \arrow[rr, "\pi"] \arrow[dr, "\alpha_{k}^{k+1}"'] && \mathcal{L}(\clh) \\
	& \mathcal{B}^{k}_{2n+1} \arrow[ur, "\varphi"'] &
\end{tikzcd}
\]
which completes the first part of the claim. \\
That, $\chi_k$ is a $C^*$-extension of $\mathcal{B}^{k}_{2n+1}$ by 
$C(\bbbt)\otimes \clk$ follows from Corollary \ref{ideal}. Moreover, we have 
$$\alpha_k^{k+1}(1)= \alpha_k^{k+1}(\sum_{i=1}^{k+1} z_i^*z_i)=\sum_{i=1}^{k+1}\,  \alpha_k^{k+1}(z_i^*z_i)= \sum_{i=1}^{k}\,z_i^*z_i=1.$$
This proves that $\chi_k$ is unital. 
To show that $\chi_k$ is homogeneous,  let $\tau^k: \mathcal{B}^{k}_{2n+1} \rightarrow Q(\bbbt)$ be its Busby invariant.  For $t_0 \in \bbbt$, 
define $\tau_{t_0}^k : \mathcal{B}^{k}_{2n+1} \rightarrow Q$ to be $ev_{t_0} \circ \tau^k$.  Assume that 
$\mathfrak{I}_{t_0}=\ker(\tau_{t_0}^k)$. We need to prove that $\mathfrak{I}_{t_0}=\left\{0\right\}$ for all $t_0 \in \bbbt$.

First, observe that $z_k \notin \mathfrak{I}_{t_0}^k$, therefore,  from Proposition \ref{P3}, we have 
$$\mathfrak{I}_{t_0}^k =\cap_{ \mathfrak{I}_{t_0}^k \subset I_{t,k}^{k,n}} I_{t,r}^{k,n}
=\mathcal{K}(L_0^k)  \otimes  I_F(\bbbt) $$ for some closed 
subset $F$ of $\bbbt$ where $I_F(\bbbt)$ is the closed ideal  of all continuous functions on $\bbbt$ vanishing on $F$. 
However, 
$$\tau_{t_0}(z_k- p^{\otimes (k-2)} \otimes t) \neq 0 \mbox{ for any } t \in \bbbt,$$ 
which further implies that $F=\bbbt$ and $\mathfrak{I}_{t_0}=\left\{0\right\}$. 

\qed
\bthm \label{K-groups-B}
For $1 \leq k \leq 2n+1$ with $k \neq n+1$, define $u_k=  z_k+1-z_kz_k^*$. 
Then one has 
\begin{IEEEeqnarray*}{rCl}
	K_0(\mathcal{B}^k_{2n+1})&=&\begin{cases}
		\langle[1]\rangle \cong \bbz & \mbox{ if } 1 \leq  k \leq n, \cr 
		\langle[1]\rangle \oplus \langle[z_{n+1}^*z_{n+1}]\rangle  \cong \bbz \oplus \bbz/2\bbz & \mbox{ if } n+2 \leq k \leq 2n+1, \cr 
	\end{cases} \\
	K_1(\mathcal{B}^k_{2n+1})&=&\langle [u_k] \rangle\cong \bbz.
\end{IEEEeqnarray*}
\ethm
\prf The claim follows by chasing the six-term exact sequence of $K$-groups obtained from the short exact sequence of $C^*$-algebras given in Lemma~\ref{ses}. We omit the details of the proof as it proceeds along the same lines as the proof of Theorem~4.3 in \cite{BhuBisSau-2026aa}. 
\qed 

\bthm \label{mQDS}
For $1 \leq k \leq 2n+1$, we have
\[
\mathcal{B}^k_{2n+1} \;\cong\; 
\begin{cases}
	\Sigma^{2(k-1)} C(\bbbt), & \text{if } 1 \leq k \leq n, \cr
	\Sigma^{2(n-1)} C(\bbbt), & \text{if } k= n+1, \cr
	\Sigma^{2(k-n-1)} \,\Sigma_2^2 \,\Sigma^{2(n-1)} C(\bbbt), & \text{if } n+2 \leq k \leq 2n. \cr 
\end{cases}
\]
In particular, 
\[
C\bigl(SO_0(2n+1)/SO_0(2n-1)\bigr) \;\cong\; \Sigma^{2(n-1)} \,\Sigma_2^2 \,\Sigma^{2(n-1)} C(\bbbt).
\]
\ethm
\prf From Lemma~\ref{ses}, $\mathcal{B}^k_{2n+1}$ appears as the middle $C^*$-algebra of a unital homogeneous $C^*$-extension $\chi_k$ of $\mathcal{B}^{k}_{2n+1}$ by 
$C(\bbbt)\otimes \clk$. Using a straightforward induction and comparing the $K$-groups, the claim follows from Proposition~3.11 in \cite{Sau-2026aa}. 
\qed 
\bcrlre
For all $q \in [0,1)$, the $C^*$-algebras $C\bigl(SO_q(2n+1)/SO_q(2n-1)\bigr)$ are isomorphic.
\ecrlre
\prf It is an immediate consequence of Theorem~\ref{mQDS} and Theorem~3.12 in \cite{Sau-2026aa}.
\qed

\bigskip
		
\noindent{\sc Surajit Biswas} (\texttt{sbiswas.maths@nitdgp.ac.in},  \texttt{241surajit@gmail.com})\\
	{\footnotesize Department of Mathematics, National Institute of Technology, Durgapur-713209, West Bengal, India}\\

\noindent{\sc Vinay Deshpande} (\texttt{deshpandevinay@iitgn.ac.in})\\
    {\footnotesize Department of Mathematics, Indian Institute of Technology, Gandhinagar, Palaj, Gandhinagar 382055, India}\\

\noindent{\sc Bipul Saurabh} (\texttt{bipul.saurabh@iitgn.ac.in},  \texttt{saurabhbipul2@gmail.com})\\
	{\footnotesize Department of Mathematics, Indian Institute of Technology, Gandhinagar, Palaj, Gandhinagar 382055, India}

\end{document}